\pdfoutput=1
\documentclass[11pt, a4paper]{article}
\usepackage[svgnames]{xcolor}
\usepackage{mathtools}
\usepackage[abbrev,lite,nobysame]{amsrefs}
\usepackage{amsthm,amsfonts,amsmath}
\usepackage{abstract}
\usepackage{hyphenat}
\usepackage{csquotes}
\usepackage{enumerate}
\usepackage[font=footnotesize,labelfont=bf]{caption}
\usepackage{subcaption}
\usepackage[colorlinks=true,linkcolor=NavyBlue,
citecolor=NavyBlue,urlcolor=NavyBlue]{hyperref}
\usepackage{cleveref}
\crefname{equation}{}{}

\usepackage{tocloft}

\usepackage[looser]{newtxtext}
\usepackage[bigdelims,varvw]{newtxmath}
\usepackage[protrusion=true, tracking=true, kerning=true,spacing=true]{microtype}
\usepackage{bm}

\usepackage{tikz}
\usepgflibrary{patterns}
\usetikzlibrary{patterns.meta,decorations.pathmorphing
	,decorations.text,positioning,shapes.geometric,arrows.meta,decorations.markings,fadings,decorations.pathreplacing,calligraphy} 
\usepackage{pgfplots}
\pgfplotsset{compat = newest}
\usepgfplotslibrary{colormaps}
\tikzfading[name=fade right,
left color=transparent!0, right color=transparent!100]
\tikzfading[name=fade left,
right color=transparent!0, left color=transparent!100]
\tikzset{>={Latex[width=3mm,length=3mm]}}

\tikzstyle{B1} = [rectangle, rounded corners, minimum width=2.2cm, minimum height=1.4cm,text centered, draw=black, fill=MidnightBlue!3, text width=3.9cm]
\tikzstyle{B11} = [rectangle, rounded corners, minimum width=2.2cm, minimum height=1.4cm,text centered, draw=black, fill=MidnightBlue!3, text width=9cm]

\tikzstyle{B0} = [rectangle, rounded corners, minimum width=2.2cm, minimum height=1.4cm,text centered, draw=black, fill=MidnightBlue!3, text width=8.8cm]
\tikzstyle{B00} = [rectangle, rounded corners, minimum width=2.2cm, minimum height=1.4cm,text centered, draw=black, fill=MidnightBlue!3, text width=11.6cm]
\tikzstyle{B000} = [rectangle, rounded corners, minimum width=2.2cm, minimum height=1.4cm,text centered, draw=black, fill=MidnightBlue!3, text width=5cm]

\tikzstyle{blue0} = [rectangle, rounded corners, minimum width=2.2cm, minimum height=1.4cm,text centered, draw=black, fill=MidnightBlue!3, text width=2.1cm]
\tikzstyle{blue} = [rectangle, rounded corners, minimum width=3.6cm, minimum height=1.4cm,text centered, draw=black, fill=MidnightBlue!3, text width=3cm]
\tikzstyle{blue1} = [rectangle, rounded corners, minimum width=6cm, minimum height=1.4cm,text centered, draw=black, fill=MidnightBlue!3, text width=6.2cm]

\tikzstyle{arrow} = [thick,->]

\tikzdeclarepattern{
	name=mylines,
	parameters={
		\pgfkeysvalueof{/pgf/pattern keys/size},
		\pgfkeysvalueof{/pgf/pattern keys/angle},
		\pgfkeysvalueof{/pgf/pattern keys/line width},
	},
	bounding box={
		(0,-0.5*\pgfkeysvalueof{/pgf/pattern keys/line width}) and
		(\pgfkeysvalueof{/pgf/pattern keys/size},
		0.5*\pgfkeysvalueof{/pgf/pattern keys/line width})},
	tile size={(\pgfkeysvalueof{/pgf/pattern keys/size},
		\pgfkeysvalueof{/pgf/pattern keys/size})},
	tile transformation={rotate=\pgfkeysvalueof{/pgf/pattern keys/angle}},
	defaults={
		size/.initial=5pt,
		angle/.initial=45,
		line width/.initial=.4pt,
	},
	code={
		\draw [line width=\pgfkeysvalueof{/pgf/pattern keys/line width}]
		(0,0) -- (\pgfkeysvalueof{/pgf/pattern keys/size},0);
	},
}

\numberwithin{equation}{section}
\theoremstyle{plain}
\newtheorem{thrm}{Theorem}[section]

\theoremstyle{definition}

\newtheorem{rmrk}[thrm]{Remark}

\theoremstyle{plain}

\newcommand{\xvec}[1]{\bm{#1}}
\newcommand{\xsym}[1]{\bm{#1}}
\newcommand{\xdop}[1]{\bm{\mathrm{#1}}}
\def\xnab{\xdop{\nabla}}

\newcommand{\xwcurl}[1]{\xdop{\nabla}\wedge{#1}}

\newcommand{\xdiv}[1]{\xdop{\nabla}\cdot{#1}}

\newcommand{\xdx}[1]{{{\rm d}#1}}
\def\xdrv#1#2{\frac{{\rm d}#1}{{\rm d}#2}}
\def\cf{\emph{cf.\/}}\def\ie{\emph{i.e.\/}}\def\eg{\emph{e.g.\/}}

\def\@Rref#1{\hbox{\rm \ref{#1}}}
\def\Rref#1{\@Rref{#1}}
\def\xCzero{{\rm C}^{0}}

\def\xCinfty{{\rm C}^{\infty}} 

\def\xH{{\rm H}}
\def\bH{{\mathbf{H}}}

\def\xHone{{\rm H}^{1}}
\def\xHtwo{{\rm H}^{2}} 

\def\xHn#1{{\rm H}^#1}

\def\xLtwo{{\rm L}^{2}}

\def\xLfour{{\rm L}^{4}}  
\def\xLinfty{{\rm L}^{\infty}} 
\def\xLn#1{{\rm L}^#1}

\def\xX{{\rm X}}

\makeatletter
\newcommand*{\toccontents}{\@starttoc{toc}}
\makeatother

\begin{document}
	\date{}
	
	\author{Manuel~Rissel\,\protect\footnote{NYU-ECNU Institute of Mathematical Sciences at NYU Shanghai, 3663 Zhongshan Road North, Shanghai, 200062, China, e-mail: \href{mailto:Manuel.Rissel@nyu.edu}{Manuel.Rissel@nyu.edu}}}
	
	\title{Global controllability of Boussinesq channel flows only through the temperature}
	
	\maketitle
	
	\begin{abstract}
		We show the global approximate controllability of the Boussinesq system with viscosity and diffusion in a planar periodic channel by using only a temperature control supported in a thin strip. At the walls, a slip boundary condition is chosen for the fluid and the normal derivative of the temperature is assumed to vanish. This contributes a first global controllability result of such type for the Boussinesq system in the presence of non-periodic boundary conditions. We resort to a small-time scaling argument to control the vorticity through a large initial temperature. Moreover, relying on the special choice of the domain, we employ J.-M. Coron's return method in order to steer the temperature without significantly impacting the vorticity.
		
		\quad
		
		\begin{center}
			\textbf{Keywords} \\ Boussinesq system, incompressible fluids, approximate controllability
			\\
			\quad
			\\
			{\bf MSC2020} \\ 35Q30, 35Q49, 76B75, 80A19, 93B05, 93C10
		\end{center}
		
		\quad
		
	\end{abstract}
	
	\section{Introduction}
	The global (large data) approximate controllability of the Boussinesq system by using solely a localized temperature control was recently tackled in \cite{NersesyanRissel2024} when the domain is the $2$D flat torus. This demonstrates for the periodic setting that, given any initial configuration, the system can approximately reach any target in arbitrary time if one acts with an appropriate physically localized force on the temperature. However, the existing literature on the controllability of the Boussinesq system is in large parts concerned with domains that have boundaries, and for these cases the question of global controllability only through the temperature has remained unanswered. The present article provides now the existence of at least one setting with boundaries for which global controllability properties can be achieved by using merely a localized temperature control. More specifically, we prove the global approximate controllability of the Boussinesq system in a straight thermally insulated periodic channel with a slip boundary condition for the fluid, and driven by an interior temperature control supported in a strip of nonzero
	width. While being far from general, this setup is not chosen just for sake of simplicity; see also \Cref{remark:otherbc}. Overall, the main challenges of the considered controllability problem are as follows: 1) the control can only act on the temperature equation; 2) the control should be localized; 3) initial and target states can be far away in the state space; 4) the temperature enters the momentum equation only in the direction of gravity; 5) the control mechanism has to be designed in accordance with the boundary conditions. Our approach rests on a scaling argument to control the vorticity trough the initial temperature. Furthermore, we employ the return method (\cf~\cite[Part 2, Chapter 6]{Coron2007}) for steering the temperature. From the viewpoint of applications, the Boussinesq system can be relevant, \eg, to the study of geophysical phenomena, Rayleigh-B\'enard convection, heating, and ventilation (\cf~\cite{AbergelTemam1990,Getling1998}).
	
	We consider viscous incompressible Newtonian flows under Boussinesq heat effects in the channel $\mathscr{C} \coloneq (-1, 1)\times\mathbb{T}$, where $\mathbb{T} \coloneq \mathbb{R} / 2\pi \mathbb{Z}$. The fluid is assumed to slip without friction along the solid boundary $\Gamma \coloneq \partial \mathscr{C} = \{-1,1\}\times\mathbb{T}$, while for the temperature we prescribe the zero Neumann boundary condition at~$\Gamma$. To influence the dynamics, only external heating/cooling, realized by an interior control $\eta \colon \omegaup\times (0,T) \longrightarrow \mathbb{R}$, can be applied in a control zone of the form
	\[
		\omegaup \coloneq [-1,1] \times [a, b], \quad 0 \leq a < b < 2\pi.
	\]
	The velocity $\xvec{u}\colon\mathscr{C}\times[0,T] \longrightarrow \mathbb{R}^2$, temperature $\theta\colon\mathscr{C}\times[0,T] \longrightarrow \mathbb{R}$, and exerted pressure $p \colon\mathscr{C}\times[0,T] \longrightarrow \mathbb{R}$ are governed by the controlled Boussinesq system
	\begin{equation}\label{equation:Boussinesq}
		\begin{gathered}
			\partial_t \xvec{u} - \nu \Delta \xvec{u} + \left(\xvec{u} \cdot \xnab\right) \xvec{u} + \xnab p = \theta \xsym{\mathscr{g}} + \xsym{\Phi}, \quad \xdiv{\xvec{u}} = 0, \\ 
			\partial_t \theta - \tau \Delta \theta + (\xvec{u}\cdot\xnab)\theta = \mathbb{I}_{\omegaup} \eta + \psi, \\
			\xvec{u}|_{\Gamma} \cdot \xvec{n} = 0, \quad (\xwcurl{\xvec{u}})|_{\Gamma} = 0, \quad \partial_{\xvec{n}}\theta |_{\Gamma} = 0,  \quad \xvec{u}(\cdot, 0) = \xvec{u}_0, \quad \theta(\cdot, 0) = \theta_0
		\end{gathered}
	\end{equation}
	where $\xvec{n}$ is the outward unit normal at $\Gamma$, the constant vector $\xsym{\mathscr{g}} \coloneq [0,1]^{\top}$ refers to gravity, $(\xvec{u}_0, \theta_0)$ are the initial states, and the external forces $(\xsym{\Phi}, \psi)$ are fixed. The evolution equations in \eqref{equation:Boussinesq} consist of a Navier--Stokes system for the fluid, where temperature enters as a force in the direction of gravity, and of a convection diffusion equation, with the fluid velocity as drift field, for the temperature; as discussed in \Cref{section:proofmainresult}, the problem \eqref{equation:Boussinesq} will be globally wellposed for the considered classes of data. The boundary condition for~$\xvec{u}$, which allows the fluid to slip at the flat boundary~$\Gamma$, is sometimes called the Lions free boundary condition (\cf~\cite[Page 129]{Lions1996}, \cite[Remark 2.4]{Temam1997}, and \cite[Section 1.5.3]{CoronMarbachSueur2020}). It can also be seen as a limiting case of the Navier slip-with-friction boundary condition dating back to \cite{Navier1823}. As made precise below in \Cref{theorem:main}, the approximate controllability of \eqref{equation:Boussinesq} in the present context essentially means that, for any $\varepsilon > 0$, initial configuration $(\xvec{u}_0, \theta_0)$, target $(\xvec{u}_T, \theta_T)$, and suitable forces $(\xsym{\Phi}, \psi)$, one can choose the control~$\eta$ in \eqref{equation:Boussinesq} such that $\|(\xvec{u}, \theta)(\cdot, T)-(\xvec{u}_T, \theta_T)\|_{\xX} < \varepsilon$, where~$\xX$ will be a product space of Sobolev type.

	When controllability properties are known for a system, it is natural to ask whether, or to which extend, such properties remain true if one limits the actions of the controls to a reduced number of components. For the $3$D Navier--Stokes equations with the no-slip boundary condition, the local null controllability with controls having two vanishing components has been shown in \cite{CoronLissy2014}; other results in this direction, including $2$D domains, are cited therein. Further, see \cite{GuerreroMontoya2018} for local exact controllability results under the assumption of (nonlinear) Navier slip boundary conditions and with controls having one vanishing component. However, there are currently no such global controllability results for the Navier--Stokes system with distributed or boundary controls, also not if slip or periodic boundary conditions are employed. And even when allowing physically localized controls to act directly in all components, the global approximate controllability of the $2$D and $3$D Navier--Stokes equations with the no-slip boundary condition is an open problem posed by J.-L. Lions (see~\cite{CoronMarbachSueurZhang2019,CoronMarbachSueur2020} and the references therein). If one admits everywhere-supported controls, but localized in frequency, the global approximate controllability (and also the Lagrangian controllability) with controls vanishing in two components has been obtained in \cite{Nersesyan2015} for the Navier--Stokes system in the $3$D flat torus. Concerning the Boussinesq system driven only by a physically localized temperature control, the local null controllability, and also the local exact controllability to trajectories, were demonstrated first in \cite{Fernandez-CaraGuerreroImanuvilovPuel2006}. An argument that requires less assumptions on the domain, but further limits the choice of reachable trajectories, is given by \cite{Carreno2012}, and the case of (nonlinear) Navier slip-with-friction boundary conditions for the velocity, and the Neumann boundary condition for the temperature, is considered in \cite{Montoya2020}. Local exact controllability results of this type could be combined with our main theorem in order to conclude the global exact controllability to the zero state, or to suitable trajectories, by using only a physically localized temperature control (\cf~\Cref{remark:r}).
	Finally, let us mention that, when the controls are allowed to act both in the velocity and temperature equations, the local exact controllability to zero, or to trajectories, has been studied, \eg, in \cite{FursikovImanuvilov1998, Guerrero2006b}, and global controllability results are obtained in \cite{Chaves-SilvaEtal2023,FernandezCaraSantosSouza2016} using the return method. 
	
	\subsection{Main result}
	
	Let~$\overline{\xH}$ consist of all $f \in \xLtwo(\mathscr{C};\mathbb{R})$ with $\int_{\mathscr{C}} f(\xvec{x}) \, \xdx{\xvec{x}} = 0$, and write $\xwcurl{\xvec{f}} = \partial_1 f_2 - \partial_2 f_1$ for the \enquote{curl} of $\xvec{f} = [f_1, f_2]^{\top}$. The notation $\|\cdot\|$ refers either to~$\|\cdot\|_{\xLtwo(\mathscr{C};\mathbb{R})}$ or to~$\|\cdot\|_{\xLtwo(\mathscr{C};\mathbb{R}^2)}$. Moreover, we define the spaces
	\begin{gather*}
		\bH \coloneq \left\{ \xvec{f} \in \xLtwo(\mathscr{C};\mathbb{R}^2) \, \, \Big| \, \, \xdiv{\xvec{f}} = 0 \mbox{ in } \mathscr{C}, \xvec{f} \cdot \xvec{n} = 0 \mbox{ at } \Gamma, \xvec{f} \cdot \xsym{\mathscr{g}} \in \overline{\xH} \right\},\\
		\xHn{m}_{0} \coloneq \xHn{m}(\mathscr{C};\mathbb{R}) \cap \xHn{1}_0(\mathscr{C};\mathbb{R}), \quad m \geq 1,\\
		\xHn{m}_{\operatorname{N}} \coloneq \left\{ f \in \xHn{m}(\mathscr{C};\mathbb{R}) \cap \overline{\xH} \, \, \Big| \, \, \partial_{\xvec{n}} f = 0 \mbox{ at } \Gamma \right\}, \quad m \geq 2,
	\end{gather*}
	endowed respectively with $\| \cdot \|$ and $\|\cdot\|_m \coloneq (\sum_{|\xsym{\alpha}| \leq m} \| \partial^{\xsym{\alpha}} \cdot\|^2)^{1/2}$. Let us also recall the elliptic regularity estimate for div-curl problems (\cf~\cite[Apendix I]{Temam2001})
	\begin{equation}\label{equation:dce}
		\|\xvec{f}\|_k \lesssim \|\xdiv{\xvec{f}}\|_{k-1} + \|\xwcurl{\xvec{f}}\|_{k - 1} + \| \xvec{f} \cdot \xvec{n}\|_{\xHn{{k-1/2}}(\Gamma; \mathbb{R})} + | \int_{\mathscr{C}} \xvec{f}(\xvec{x}) \cdot \xsym{\mathscr{g}} \, \xdx{\xvec{x}}|,
	\end{equation}
	where $a \lesssim b$ means $a \leq C b$ and $C > 0$ shall always refer to an absolute constant. 
	The last term on the right-hand side in \eqref{equation:dce} appears since $\mathscr{C}$ is doubly-connected and its de Rham first cohomology space is spanned by $\xsym{\mathscr{g}} = [0,1]^{\top}$, which obeys
	\begin{equation*}\label{equation:dR}
		\xwcurl{\xsym{\mathscr{g}}} = 0, \quad \xdiv{\xsym{\mathscr{g}}} = 0, \quad \xsym{\mathscr{g}} |_{\Gamma} \cdot \xvec{n} = 0.
	\end{equation*}
	Throughout, we assume that the Lebesgue measure is normalized so that $\smallint_{\mathscr{C}} \, \xdx{\xvec{x}} = 1$.
	
	\begin{thrm}[Main result]\label{theorem:main}
		The system \eqref{equation:Boussinesq} is globally approximately controllable by using only a temperature control. That is, for arbitrary
		\begin{itemize}
			\item viscosity $\nu > 0$, diffusivity $\tau > 0$, control time $T > 0$, accuracy $\varepsilon > 0$,
			\item states $(\xvec{u}_0, \theta_0) \in \bH\times\overline{\xH}$ and $(\xvec{u}_T, \theta_T) \in \bH\times\xHn{2}_{\operatorname{N}}$ with $\xwcurl{\xvec{u}_T} \in \xHn{1}_{0}$,
			\item forces $(\xsym{\Phi}, \psi) \in \xHone((0, T); \bH\times\overline{\xH}) \cap \xLtwo((0, T); \xHone(\mathscr{C};\mathbb{R}^{2+1}))$,
		\end{itemize}
		there exists a control $\eta \in \xCinfty(\mathscr{C}\times[0,T];\mathbb{R})$ with $\operatorname{supp}(\eta) \subset \omegaup\times(0,T)$ such that the unique
		solution $(\xvec{u}, \theta)$ to \eqref{equation:Boussinesq} satisfies
		\[
			\|\xvec{u}(\cdot, T) - \xvec{u}_T\|_{2} + \|\theta(\cdot, T) - \theta_T\|_{2}  < \varepsilon.
		\]
	\end{thrm}
	
	\begin{rmrk}\label{remark:r0}
		Following the lines of \cite{Temam2001}, if \smash{$(\xvec{u}_0, \theta_0) \in \bH\times\xHn{2}_{\operatorname{N}}$} with \smash{$\xwcurl{\xvec{u}_0} \in \xHn{1}_{0}$} in \Cref{theorem:main}, then $(\xvec{u}, \theta)\in\smash{\xLinfty((0,T);\xHtwo(\mathscr{C};\mathbb{R}^{2+1}))\cap\xLtwo((0,T);\xHn{3}(\mathscr{C};\mathbb{R}^{2+1}))}$. Further, one can then show that $(\xvec{u}, \theta) \in \smash{\xCzero([0,T];\xHtwo(\mathscr{C};\mathbb{R}^{2+1}))}$ by using the abstract argument from~\cite[Theorem 3.1]{Temam1982}.
	\end{rmrk}
	
	\begin{rmrk}\label{remark:r}
		Due to the parabolic smoothing effects for Leray-Hopf type solutions, as detailed in \cite[Lemma 2.1]{Chaves-SilvaEtal2023} (see also \cite[Proof of Lemma 9]{CoronMarbachSueur2020} or \cite[Appendix A]{LiaoSueurZhang2022}), one can use \Cref{theorem:main} without external forces to obtain global exact controllability results with the help of known local exact ones, \eg, from \cite{Montoya2020}; we refer to \cite[Section 2.6 and Section 5]{CoronMarbachSueur2020} for the details of such an argument. By density, \Cref{theorem:main} likewise provides approximate controllability in $\xLtwo\times\xLtwo$ or $\xHone\times\xHone$; controllability in higher norms could also be achieved for suitable forces. 
	\end{rmrk}

	\begin{rmrk}\label{remark:otherbc}
		Regarding other boundary conditions, \eg, modeling friction at the walls, difficulties arise in the context of \Cref{theorem:HDS}, where viscous boundary layers would enter the proof. For the Navier--Stokes system with the Navier slip-with-friction condition, it was shown in \cite{CoronMarbachSueur2020} how these boundary layers can be dissipated sufficiently; this analysis has been extended in \cite{Chaves-SilvaEtal2023} to the Boussinesq system in the presence of both velocity and temperature controls. For the no-slip boundary condition, and even with controls in both the velocity and the temperature equations, the question of global approximate controllability parallels the open problem due to J.-L. Lions mentioned in the introduction. The shape of the domain is also crucially used here, as it facilitates a suitable return method trajectory for \eqref{equation:Boussinesq} controlled through the temperature (\cf~\Cref{remark:return}).
	\end{rmrk}

	\section{Proof of the main result}\label{section:proofmainresult}
	The proof of \Cref{theorem:main} involves two key ingredients. 1) The approximate controllability of the vorticity in a short time via a large control acting through the initial condition for the temperature; see~\Cref{theorem:LCST} in \Cref{subsection:vort}.~2) A version of the return method, implemented in \Cref{subsection:temp} and culminating in \Cref{theorem:HDS}. All pieces of the argument are finally combined in \Cref{subsection:conclusion}.
	
	\subsection{Remarks on the Boussinesq system}\label{subsection:remarksbp}
	
	As \eqref{equation:Boussinesq} accounts for viscosity and diffusion, its global wellposedness and regularity theories parallel those for the $2$D Navier--Stokes system; see, \eg, \cite{Temam2001,FoiasManleyTemam1987,Temam1997}, and also \cite{Chaves-SilvaEtal2023} for further references and a description of (controlled) Leray-Hopf weak solutions.

	Let us assume \smash{$(\xvec{u}_0, \theta_0) \in \bH\times\xHn{2}_{\operatorname{N}}$} with \smash{$\xwcurl{\xvec{u}_0} \in \xHn{1}_{0}$} . The vorticity formulation of~\eqref{equation:Boussinesq} is obtained by acting with \enquote{$\xwcurl{}$} on the velocity equation; \ie,
	\begin{equation}\label{equation:vorticity}
		\begin{gathered}
			\partial_t w - \nu \Delta w + (\xvec{u}\cdot\xnab)w = \partial_1 \theta + \varphi, \quad w(\cdot, 0) = w_0,\\
			\partial_t \theta - \tau \Delta \theta + (\xvec{u}\cdot\xnab)\theta = \mathbb{I}_{\omegaup} \eta + \psi, \quad \theta(\cdot, 0) = \theta_0, \\
			\xwcurl{\xvec{u}} = w, \quad \xdiv{\xvec{u}} = 0, \quad \int_{\mathscr{C}} \xvec{u}(\xvec{x},t) \cdot \xsym{\mathscr{g}} \, \xdx{\xvec{x}} = \int_0^t \int_{\mathscr{C}} \theta(\xvec{x},s) \, \xdx{\xvec{x}} \xdx{s},\\
			\xvec{u}|_{\Gamma} \cdot \xvec{n} = 0, \quad (\xwcurl{\xvec{u}})|_{\Gamma} = 0, \quad \partial_{\xvec{n}}\theta |_{\Gamma} = 0,
		\end{gathered}
	\end{equation}
	where $w_0 = \xwcurl{\xvec{u}_0}$, $\varphi = \xwcurl{\xsym{\Phi}}$, and $\psi$ are given as in \Cref{theorem:main}. If $\mathbb{I}_{\omegaup} \eta (\cdot, t)$ is average-free for $0 \leq t \leq T < +\infty$, then also $\theta(\cdot, t)$. Moreover, due to~\eqref{equation:dce}, the average of $\xvec{u}(\cdot, t) \cdot \xsym{\mathscr{g}}$ must be specified for all~$t$.
	Further, assuming $\mathbb{I}_{\omegaup} \eta$ to be sufficiently regular, we denote by
	\[
		S(w_0, \theta_0, \varphi, \mathbb{I}_{\omegaup} \eta + \psi) \coloneq (w, \theta) \in \xCzero([0,T]; \xHn{1}_{0}\times\xHn{{2}}_{\operatorname{N}}) \cap \xLtwo((0,T); \xHn{{2}}_{0}\times\xHn{{3}}_{\operatorname{N}})
	\]
	the respective solution to \eqref{equation:vorticity}, and by $S_t = (S^1_t, S^2_t)$ its restriction at time $t$; \ie,
	\begin{equation}\label{equation:restr}
		S^1_t(w_0, \theta_0, \varphi, \mathbb{I}_{\omegaup} \eta + \psi) \coloneq w(\cdot, t), \quad
		S^2_t(w_0, \theta_0, \varphi, \mathbb{I}_{\omegaup} \eta + \psi) \coloneq \theta(\cdot, t).
	\end{equation}

	\subsection{Steering the vorticity through the initial temperature}\label{subsection:vort}
	
	The next theorem allows to steer the vorticity in a short time by means of a large initial temperature. Our proof develops the ansatz which has been introduced in \cite{NersesyanRissel2024} for the torus case (see also \cite{BoulvardGaoNersesyan2023}). Due to boundary effects, we can only reach approximately the vorticity states of the form $w_0 - \partial_1\xi$, for a sufficiently regular profile~$\xi$ which satisfies $\partial_{1}\xi|_{\Gamma} = 0$ and the less intuitive condition $\partial_{111}\xi|_{\Gamma} = 0$.
	
	\begin{thrm}\label{theorem:LCST}
		Let~$T > 0$, $\xi \in \xCinfty(\mathscr{C}; \mathbb{R})\cap\xHn{2}_{\operatorname{N}}$ with $\partial_{111}\xi|_{\Gamma} = 0$, $(w_0, \theta_0) \in \xHn{{2}}_{0} \times \xHn{{2}}_{\operatorname{N}}$,~forces $\varphi = \xwcurl{\xsym{\Phi}}$ and $\psi$ given by \Cref{theorem:main}, and~$\varepsilon > 0$. There exists $\delta > 0$ with
		\begin{equation}\label{equation:initialdatacontrolresult}
			\|S^1_{\delta}(w_0 , \theta_0 - \delta^{-1} \xi, \varphi, \psi) - (w_0 - \partial_1\xi)\|_1 < \varepsilon.
		\end{equation}
	\end{thrm}
	\begin{proof} 
		Thanks to the assumptions on $\xi$, it follows for any $\delta \in (0,1)$ that the pair
		\begin{equation}\label{equation:defwtd}
			(w_{\delta}, \theta_{\delta}) \coloneq S_{\delta}(w_0, \theta_0 - \delta^{-1} \xi, \varphi, \psi) + (0, \delta^{-1} \xi)
		\end{equation}
		is well-defined and solves
		\begin{equation}\label{equation:AMControl}
			\begin{gathered}
				\partial_t w_{\delta} - \nu \Delta w_{\delta} + \left(\xvec{u}_{\delta} \cdot \xnab\right) w_{\delta} = \partial_1 (\theta_{\delta} - \delta^{-1} \xi) + \varphi, \quad w_{\delta}(\cdot, 0) = w_0 \\
				\partial_t \theta_{\delta} - \tau \Delta (\theta_{\delta} - \delta^{-1} \xi) + (\xvec{u}_{\delta} \cdot \xnab) (\theta_{\delta} - \delta^{-1} \xi) = \psi, \quad \theta_{\delta}(\cdot, 0) = \theta_0,\\
				\xwcurl{\xvec{u}_{\delta}} = w_{\delta}, \quad \xdiv{\xvec{u}_{\delta}} = 0, \quad \int_{\mathscr{C}} \xvec{u}_{\delta}(\xvec{x}, t) \cdot \xsym{\mathscr{g}} \, \xdx{\xvec{x}} = 0,  \\
				\xvec{u}_{\delta}|_{\Gamma} \cdot \xvec{n} = 0, \quad w_{\delta}|_{\Gamma} = 0, \quad \partial_{\xvec{n}}\theta_{\delta} |_{\Gamma} = 0.
			\end{gathered}
		\end{equation}
		In particular, as mentioned in \Cref{subsection:remarksbp}, the temperature~$ \theta_{\delta}(\cdot, t)$ is average-free for all $t \in [0, \delta]$; thus, by \eqref{equation:defwtd}, in \eqref{equation:AMControl} one actually has
		\[
			\int_{\mathscr{C}} \xvec{u}_{\delta}(\xvec{x}, t) \cdot \xsym{\mathscr{g}} \, \xdx{\xvec{x}} = 	0 = \int_0^t \int_{\mathscr{C}} \theta_{\delta}(\xvec{x},s) \, \xdx{\xvec{x}} \xdx{s}, \quad t \in [0, \delta].
		\]
		\paragraph{Step 1. Ansatz.} To establish $\lim_{\delta \to 0}  w_{\delta}(\cdot, \delta) = (w_0 - \partial_1\xi)$ in $\xHn{{1}}_{0}$, the following ansatz is made such that $\delta^{-1} \partial_1 \xi$, $\tau \delta^{-1} \Delta  \xi$, and $\delta^{-1} (\xvec{u}_{\delta} \cdot \xnab) \xi$ are canceled in the remainder estimates. More precisely, for each $\delta \in (0, 1)$ we define the remainders
		\begin{equation}\label{equation:qrd}
			\begin{gathered}
				q_{\delta}(\xvec{x}, t) \coloneq w_{\delta}(\xvec{x}, t) - w_0(\xvec{x}) + \delta^{-1} t \partial_1 \xi (\xvec{x}),\\
				r_{\delta}(\xvec{x}, t) \coloneq \theta_{\delta}(\xvec{x}, t) - \theta_0(\xvec{x}) 
				+ \delta^{-1}t \tau \Delta\xi(\xvec{x}) -  \left(\xvec{U}_{\delta,\xi}(\xvec{x} ,t) \cdot \xnab \right) \xi(\xvec{x}),
			\end{gathered}
		\end{equation}
		where $\xvec{U}_{\delta,\xi}$ is uniquely determined as the solution to
		\begin{equation}\label{equation:Udeltaxi}
			\begin{gathered}
				\xwcurl{\xvec{U}_{\delta,\xi}}(\cdot, t) = \delta^{-1}t \left(w_0 - \frac{\delta^{-1} t \partial_1 \xi}{2}\right), \quad \xdiv{\xvec{U}_{\delta,\xi}} = 0, \\ 
				\xvec{U}_{\delta,\xi}|_{\Gamma} \cdot \xvec{n} = 0, \quad \int_{\mathscr{C}} \xvec{U}_{\delta,\xi}(\xvec{x}, \cdot) \cdot \xsym{\mathscr{g}} \, \xdx{\xvec{x}} = 0.
			\end{gathered}
		\end{equation}
		In view of \eqref{equation:qrd}, the convergence \eqref{equation:initialdatacontrolresult} will follow after showing that
		\begin{equation}\label{equation:convqd}
			\lim\limits_{\delta \to 0}  \|q_{\delta}(\cdot, \delta)\|_{1} = 0.
		\end{equation}
		To this end, we consider the problems satisfied by $q_{\delta}$ and $r_{\delta}$, which are derived by taking $\partial_t$ in \eqref{equation:qrd} and inserting \eqref{equation:AMControl}; namely,
		\begin{multline}\label{equation:stlce1}
			\partial_t q_{\delta} - \nu \Delta q_{\delta} + (\xvec{Q}_{\delta} \cdot \xnab) q_{\delta} = \varphi + \partial_1 (\theta_{\delta} - r_{\delta}) + \partial_1 r_{\delta} - \nu \Delta (q_{\delta} - w_{\delta}) \\ 
			+ ((\xvec{u}_{\delta} - \xvec{Q}_{\delta}) \cdot \xnab) (q_{\delta} - w_{\delta}) + (\xvec{Q}_{\delta} \cdot \xnab) (q_{\delta} - w_{\delta}) + ((\xvec{Q}_{\delta} - \xvec{u}_{\delta}) \cdot \xnab) q_{\delta}
		\end{multline}
		and
		\begin{multline}\label{equation:stlce12}
			\partial_t r_{\delta} - \tau\Delta r_{\delta} + (\xvec{Q}_{\delta} \cdot \xnab) r_{\delta} = \psi - \tau\Delta (r_{\delta} - \theta_{\delta}) + (\xvec{Q}_{\delta} \cdot \xnab) (r_{\delta} - \theta_{\delta})  \\
			+ ((\xvec{Q}_{\delta} - \xvec{u}_{\delta}) \cdot \xnab) r_{\delta} + ((\xvec{Q}_{\delta} - \xvec{u}_{\delta}) \cdot \xnab) (\theta_{\delta} - r_{\delta})
			+ \delta^{-1} (\xvec{Q}_{\delta} \cdot \xnab) \xi,
		\end{multline}
		where $\xvec{Q}_{\delta}$ is the solution to
		\begin{equation*}\label{equation:bdrv0}
			\xwcurl{\xvec{Q}_{\delta}} = q_{\delta}, \quad \xdiv{\xvec{Q}_{\delta}} = 0, \quad \xvec{Q}_{\delta}|_{\Gamma} \cdot \xvec{n} = 0, \quad
			\int_{\mathscr{C}} \xvec{Q}_{\delta}(\xvec{x},\cdot) \cdot \xsym{\mathscr{g}} \, \xdx{\xvec{x}} = 0.
		\end{equation*}
		By the hypotheses stated in \Cref{theorem:LCST}, in particular $\partial_{1}\xi|_{\Gamma} = \partial_{111}\xi|_{\Gamma} = 0$, and due to the definitions in \eqref{equation:qrd} and~\eqref{equation:Udeltaxi}, the initial and boundary values of the remainders are
		\begin{equation*}\label{equation:bdrv}
			q_{\delta}(\cdot, 0) = 0, \quad r_{\delta}(\cdot, 0) = 0, \quad q_{\delta}|_{\Gamma} = 0, \quad \partial_{\xvec{n}} r_{\delta}|_{\Gamma} = 0.
		\end{equation*}
		
		\paragraph{Step 2. Estimates.} We utilize \eqref{equation:dce}, and the Sobolev embeddings $\xHone(\mathscr{C}; \mathbb{R}) \subset \xLfour(\mathscr{C}; \mathbb{R})$ and $\xHtwo(\mathscr{C}; \mathbb{R}) \subset \xLinfty(\mathscr{C}; \mathbb{R})$. First, the equation for~$q_\delta$ in \eqref{equation:stlce1} is multiplied by~$q_{\delta}$, and in a second step by~$-\Delta q_{\delta}$. Then, we integrate over $\mathscr{C}\times(0,t)$ with $t \in [0, \delta]$. Because $\xvec{u}_{\delta}$ and $\xvec{Q}_{\delta}$ are divergence free, an also due to the known boundary values of~$q_{\delta}$,~$\xvec{Q}_{\delta}$, and~$r_{\delta}$, integration by parts yields
		\begin{equation*}
			\begin{multlined}
				\|q_{\delta}(\cdot, t)\|^2 + \int_0^t \| q_{\delta}(\cdot, s) \|^2_1 \, \xdx{s} \lesssim \int_0^{t} \| \varphi(\cdot, s) \|^2 \, \xdx{s}  + \int_0^{t} \| (\theta_{\delta} - r_{\delta})(\cdot, s) \|_1^2 \, \xdx{s} \\
				\begin{aligned}
					& + \int_0^{t} \| (q_{\delta} - w_{\delta})(\cdot, s) \|_2^2 \, \xdx{s} + \int_0^{t} \| (q_{\delta} - w_{\delta})(\cdot, s) \|_1^4 \, \xdx{s} + \int_0^t \| q_{\delta} (\cdot, s) \|^4 \, \xdx{s}\\
					& + \int_0^t \left(\| q_{\delta} (\cdot, s) \|^2 + \| r_{\delta}(\cdot, s) \|_1^2\right) \, \xdx{s}  \eqcolon I_{1,1} + \dots + I_{1,6}
				\end{aligned}
			\end{multlined}
		\end{equation*}
		and\allowdisplaybreaks
		\begin{multline*}
			\|q_{\delta}(\cdot, t)\|^2_1 + \int_0^t \| \Delta q_{\delta}(\cdot, s) \|^2 \, \xdx{s} \\
			\begin{aligned}
				& \lesssim \ell 
				\int_0^t \| q_{\delta}(\cdot, s) \|_2^2 \, \xdx{s} + \ell^{-1} \int_0^{t} \| \varphi(\cdot, s) \|^2 \, \xdx{s} + \ell^{-1} \int_0^{t} \| (\theta_{\delta} - r_{\delta})(\cdot, s) \|_1^2 \, \xdx{s} \\
				& \quad + \ell^{-1} \int_0^{t} \| (q_{\delta} - w_{\delta})(\cdot, s) \|_2^2 \, \xdx{s} + \int_0^{t} \| (q_{\delta} - w_{\delta})(\cdot, s) \|_2^4 \, \xdx{s} \\
				& \quad + \int_0^t \| q_{\delta}(\cdot, s) \|_1^2 \, \xdx{s} + \ell^{-1} \int_0^t \| r_{\delta}(\cdot, s) \|_1^2 \, \xdx{s} + \int_0^t \| q_{\delta}(\cdot, s) \|_1^4 \, \xdx{s}\\
				& \eqcolon I_{2,1} + \dots + I_{2,8}
			\end{aligned}
		\end{multline*}
		for any $\ell > 0$. Regarding \eqref{equation:stlce12}, similar considerations as above lead to
		\begin{multline*}
			\|r_{\delta}(\cdot, t)\|^2 + \int_0^t \| r_{\delta}(\cdot, s) \|^2_1 \, \xdx{s} \\
			\begin{aligned}
				& \lesssim \int_0^{t} \| \psi(\cdot, s) \|^2 \, \xdx{s} + \int_0^{t} \| (r_{\delta} - \theta_{\delta})(\cdot, s) \|_2^2 \, \xdx{s} + \int_0^{t} \| (r_{\delta} - \theta_{\delta})(\cdot, s) \|_1^4 \, \xdx{s} \\
				& \quad  + \int_0^{t} \| (q_{\delta} - w_{\delta})(\cdot, s) \|_1^4 \, \xdx{s} + \int_0^t \| r_{\delta}(\cdot, s) \|^4 \, \xdx{s} \\
				& \quad  + \delta^{-1} \|\xi\|_3 \int_0^t  \left( \|r_{\delta}(\cdot, s) \|^2  + \| q_{\delta}(\cdot, s) \|^2\right)  \, \xdx{s}\\
				& \eqcolon J_{1,1} + \dots + J_{1,6}
			\end{aligned}
		\end{multline*}
		and
		\begin{multline*}
			\|r_{\delta}(\cdot, t)\|^2_1 + \int_0^t \| r_{\delta}(\cdot, s) \|^2_2 \, \xdx{s} \\
			\begin{aligned}
				& \lesssim  \ell 
				\int_0^t \| r_{\delta}(\cdot, s) \|_2^2 \, \xdx{s} + \ell^{-1} \int_0^{t} \| \psi(\cdot, s) \|^2 \, \xdx{s} \\
				& \quad + \int_0^{t} \| (r_{\delta} - \theta_{\delta})(\cdot, s) \|_2^4 \, \xdx{s} + \ell^{-1} \int_0^{t} \| (r_{\delta} - \theta_{\delta})(\cdot, s) \|_2^2 \, \xdx{s} \\
				& \quad + \int_0^{t} \| (q_{\delta} - w_{\delta})(\cdot, s) \|_2^4 \, \xdx{s} + \int_0^t \| r_{\delta}(\cdot, s) \|_1^4 \, \xdx{s} \\
				& \quad + \int_0^t \| q_{\delta}(\cdot, s) \|^4 \, \xdx{s} + \delta^{-1} \|\xi\|_4 \int_0^t \left(\| r_{\delta}(\cdot, s) \|_1^2  + \| q_{\delta}(\cdot, s) \|^2 \right) \, \xdx{s} \\
				& \eqcolon J_{2,1} + \dots + J_{2,8}.
			\end{aligned}
		\end{multline*}
		All of the previous estimates are combined, while fixing $\ell > 0$ sufficiently small (independently of~$\delta$) such that the terms~$I_{2,1}$ and~$J_{2,1}$ are absorbed by the resulting left-hand side; \ie,
		\begin{equation}\label{equation:ee5}
			\begin{gathered}
				\|q_{\delta}(\cdot, t)\|^2_1 + \|r_{\delta}(\cdot, t)\|^2_1 + (1 - \ell) \int_0^t \left(\| q_{\delta}(\cdot, s) \|_2^2 + \| r_{\delta}(\cdot, s) \|_2^2\right)  \, \xdx{s} \\
				\lesssim \sum_{i=1}^6 (I_{1,i} + J_{1,i}) + \sum_{i=2}^{8} (I_{2,i} + J_{2,i}),
			\end{gathered}
		\end{equation}
		where we used that
		\begin{equation}\label{equation:LaplH2}
			\|\cdot\|^2_{\xHn{2}(\mathscr{C}; \mathbb{R})} \lesssim \|\cdot\|^2_{\xLtwo(\mathscr{C}; \mathbb{R})} + \|\Delta \cdot\|^2_{\xLtwo(\mathscr{C}; \mathbb{R})} .
		\end{equation}
		Moreover, for any $a \geq 1$ and $l \in \mathbb{N}$, it follows from \cref{equation:dce,equation:qrd,equation:Udeltaxi} that
		\begin{equation*}
			\begin{gathered}
				\|(q_{\delta} - w_{\delta})\|_{\xLn{{a}}((0, \delta); \xHn{{l}})}^{a} \lesssim \delta \left(1 + \|w_0\|_{l}^{a} + \| \xi \|_{l+1}^{a}\right),\\
				\|(r_{\delta} - \theta_{\delta})\|_{\xLn{{a}}((0, \delta); \xHn{{l}})}^{a} \lesssim \delta \left(1 + \|w_0\|_{l-1}^{2a} + \|\theta_0\|_{l}^a + \| \xi \|_{l+3}^{2a}\right),
			\end{gathered}
		\end{equation*}
		which implies
		\begin{equation}\label{equation:asympta}
			\|(q_{\delta} - w_{\delta})\|_{\xLn{{a}}((0, \delta); \xHn{{2}})}^{a} + \|(r_{\delta} - \theta_{\delta})\|_{\xLn{{a}}((0, \delta); \xHn{{3}})}^{a} \lesssim \delta.
		\end{equation}
		Because the forces $(\varphi, \psi)$ are fixed, it follows with the help of \eqref{equation:asympta} that
		\begin{equation*}
			\lim_{\delta \to 0} \left( \sum_{i=1}^4 (I_{1,i} + J_{1,i}) + \sum_{i=2}^5 (I_{2,i} + J_{2,i})\right) = 0.
		\end{equation*}
		Also, let us emphasize that
		\begin{equation*}\label{equation:alpha}
			\begin{gathered}
				J_{1,6} + J_{2, 8} \lesssim \int_0^{\delta} \alpha_{\delta} \left(\|q_{\delta}(\cdot, s)\|_1^2 + \|r_{\delta}(\cdot, s)\|_1^2\right)\, \xdx{s}, \quad
				\alpha_{\delta} \coloneq \delta^{-1}\max\{1, \|\xi\|_4^2\},
			\end{gathered}
		\end{equation*}
		where $\smash{\smallint_0^{\delta} \alpha_{\delta} \, \xdx{s} = \max\{1, \|\xi\|_4^2\}}$.
		Now, the remaining integrals $I_{1,i}$, $I_{2,k}$, $J_{1,i}$, and~$J_{2,k}$, for $i > 4, k > 5$, are good terms for applying Gr\"onwall's inequality in \eqref{equation:ee5}, which then yields the existence of $c_{\delta} > 0$ with $\lim_{\delta \to 0} c_{\delta} = 0$ and such that
		\begin{multline*}
			\|q_{\delta}(\cdot, t)\|^2_1 + \|r_{\delta}(\cdot, t)\|^2_1\\
			\leq \left(c_{\delta} + C\int_0^t \left(\| q_{\delta}(\cdot, s) \|_1^4 + \| r_{\delta}(\cdot, s) \|_1^4\right) \, \xdx{s} \right) \operatorname{exp}({C\max\{1, \|\xi\|_4^2\}})
		\end{multline*}
		for all $t \in [0, \delta]$. Subsequently, we rename $c_{\delta}\operatorname{exp}({C\max\{1, \|\xi\|_4^2\}})$ again as~$c_{\delta}$ and absorb $\operatorname{exp}({C\max\{1, \|\xi\|_4^2\}})$ in the absolute constant $C > 0$. Then, we define
		\[
			f(t) \coloneq c_{\delta}  + C \int_0^t \left( \| q_{\delta}(\cdot, s) \|_1^4 + \| r_{\delta}(\cdot, s) \|_1^4\right) \, \xdx{s}.
		\]
		Hence, one has $f'/f^2 \leq C$ and thus $f(\delta) \leq c_{\delta} (1-c_{\delta} \delta C)^{-1} \longrightarrow 0$ as $\delta \longrightarrow 0$, which implies \eqref{equation:convqd}.
	\end{proof}
	
	\subsection{Controlling the temperature} \label{subsection:temp}
	We define a constant-in-$\xvec{x}$ vector field whose integral curves all cross the control region~$\omegaup$. This construction is similar to \cite{NersesyanRissel2024,NersesyanRissel2022} and shall provide a return method type flow (\cf~\Cref{remark:return}) with special structure that facilitates the proof of \Cref{theorem:HDS}. To begin with, we fix $0 < H_1 < H_2 < 2\pi$ such that $[-1, 1] \times [H_1, H_2] \subset \omegaup$, and we further choose a possibly large~$K \in \mathbb{N}$ with $l_K \coloneq 8\pi/3K <  H_2 - H_1/3$.
	Then, the channel $\mathscr{C}$ is covered by the overlapping rectangles
	\[
		\mathcal{O}_i \coloneq (-2, 2) \times \left(\frac{3(i-1)l_{K}}{4}, \frac{3(i-1)l_{K}}{4} + l_{K}\right), \quad i \in \{1,\dots,K\},
	\] 
	which are vertical translations of the reference rectangle
	\[	
		\mathcal{O} \coloneq (-2, 2) \times (H_1+l_K, H_1+2l_K) \subset \omegaup.
	\]
	Moreover, a cutoff function $\chi \in \xCinfty(\mathscr{C};[0,1])$ with~$\operatorname{supp}(\chi) \subset \mathcal{O}$, and which only depends on~$x_2$, is given via
	\begin{equation}\label{equation:chi}
		\chi(\xvec{x}) \coloneq \widetilde{\chi}(x_2-H_1-l_K),
	\end{equation}
	where $\widetilde{\chi} \in \xCinfty(\mathbb{T};[0,1])$ is any profile satisfying
	\begin{equation}\label{equation:chip}
		\begin{gathered}
			\operatorname{supp}(\widetilde{\chi}) \subset (0, l_K), \quad \forall x \in (0,l_K/4)\colon  \widetilde{\chi}(x) + \widetilde{\chi}(x + 3l_K/4) = 1,\\
			\widetilde{\chi}(s) = 1 \iff s \in \left[l_K/4, 3l_K/4\right].
		\end{gathered}
	\end{equation}
	
	To state the next result, the reference time interval~$[0,1]$ is partitioned equidistantly by
	\begin{equation}\label{equation:edp}
		0 < t^0_c < t^1_a < t^1_b < t^1_c < t^2_a < t^2_b < t^2_c < \dots < t^K_a < t^K_b < t^K_c < 1.
	\end{equation}
	
	\begin{thrm}\label{theorem:convection}
		There exists $\overline{\xvec{y}} = [0, \overline{y}_2]^{\top} \in \xCinfty_0((0,1); \mathbb{R}^2)$, denoting by~$\xsym{\mathcal{Y}}$ its flow obtained via $\smash{\xdrv{\xsym{\mathcal{Y}}}{t}(\xvec{x}, s, t) = \overline{\xvec{y}}(t)}$ and $\xsym{\mathcal{Y}}(\xvec{x}, s, s) = \xvec{x}$, such that one has the properties:
		\begin{enumerate}[P1)]
			\item Supported in $(t^0_c, t^K_c)$: 	$\forall t \in [0, t^0_c] \cup [t^K_c, 1]\colon \overline{\xvec{y}}(t) = \xvec{0}$;
			\item Closed integral curves: $\forall \xvec{x} \in \mathscr{C} \colon \xsym{\mathcal{Y}}(\xvec{x},0,1) = \xvec{x}$;
			\item Stationary visits of $\mathcal{O}$: $\forall i \in \{1,\dots,K\} \colon \,\xsym{\mathcal{Y}}(\mathcal{O}_i, 0, [t^i_a, t^i_b]) = \mathcal{O}$.
		\end{enumerate}
	\end{thrm}
	\begin{proof}
		The construction from \cite{NersesyanRissel2024,NersesyanRissel2022} works here as well. Recalling that $\overline{t} \coloneq t^0_c$ is the width of the partition in~\eqref{equation:edp}, one selects $(\beta_i)_{i\in\{1,\dots,K\}} \subset \xCinfty_0((0, \overline{t});\mathbb{R})$ such that $\mathcal{O}_i + \xsym{\mathscr{g}}\smallint_0^{\overline{t}} \beta_i(s) \, \xdx{s} = \mathcal{O}$.
		Then, one defines
		\[
			\overline{\xvec{y}}(t) \coloneq \begin{cases}
				\xsym{0} & \mbox{ if } t \in [0, t^0_c] \cup [t^K_c, 1],\\
				\xvec{h}_i(t-(3i-2)\overline{t}) & \mbox{ if } t \in (t^{i-1}_c, t^i_c) \mbox{ for } i \in \{1,\dots, K\},
			\end{cases}
		\]
		where $\xvec{h}_i(t) = \beta_i(t) \xsym{\mathscr{g}}$ if $ t \in [0, \overline{t}]$, $\xvec{h}_i(t) = \xsym{0}$ if $ t \in (\overline{t}, 2\overline{t})$, and $\xvec{h}_i(t) = -\beta_i(t - 2\overline{t}) \xsym{\mathscr{g}}$ if $t \in [2\overline{t}, 3\overline{t}]$.
	\end{proof}
		
	\begin{rmrk}\label{remark:return}
		Any $\overline{\xvec{y}}$ from \Cref{theorem:ControlLin} is a reference velocity in the spirit of the return method as developed for the incompressible Euler and Navier--Stokes equations (\cf~\cite[Part 2, Chapter 6]{Coron2007}). Indeed, it holds $\overline{\xvec{y}}(0) = \overline{\xvec{y}}(1) = 0$, and the integral curves of $\overline{\xvec{y}}$ all cross the control zone. Moreover, a special trajectory for a controlled inviscid Boussinesq system is given by $(\overline{\xvec{u}}, \overline{\theta}) \coloneq (\overline{\xvec{y}}, \overline{y}_2' \chi / \smallint_{\mathscr{C}} \chi(\xvec{z}) \, \xdx{\xvec{z}})$, which solves the controllability problem
		\begin{gather*}
			\partial_t \overline{\xvec{u}} + (\overline{\xvec{u}} \cdot \xnab)\overline{\xvec{u}} + \xnab \overline{p} = \overline{\theta} \xsym{\mathscr{g}}, \quad \xdiv{\overline{\xvec{u}}} = 0, \quad \partial_t\overline{\theta} + (\overline{\xvec{u}} \cdot \xnab)\overline{\theta} = \mathbb{I}_{\omegaup}\overline{\eta}, \\
			\overline{\xvec{u}}|_{\Gamma} \cdot \xvec{n} = 0, \quad (\xwcurl{\overline{\xvec{u}}})|_{\Gamma} = 0, \quad \partial_{\xvec{n}}\overline{\theta} |_{\Gamma} = 0, \\ 
			\overline{\xvec{u}}(\cdot, 0) = \overline{\xvec{u}}(\cdot, 1) = 0, \quad \overline{\theta}(\cdot, 0) = \overline{\theta}(\cdot, 1) = 0,
		\end{gather*}
		where
		\[
			\overline{p}(\xvec{x},t) = \int_0^{x_2}\left(\frac{\overline{y}_2'(t) \chi(s)}{\int_{\mathscr{C}} \chi(\xvec{z}) \, \xdx{\xvec{z}}} - \overline{y}_2'(t)\right) \, \xdx{s}, \quad \overline{\eta}(\xvec{x},t) = \frac{(\overline{y}_2''\chi + \overline{y}_2'(\overline{\xvec{y}} \cdot \xnab)\chi)(\xvec{x}, t)}{\int_{\mathscr{C}} \chi(\xvec{z}) \, \xdx{\xvec{z}}}.
		\]
	\end{rmrk}

	Next, we demonstrate the approximate controllability of a linear transport problem with drift $\overline{\xvec{y}}$ and a smooth localized control. Since the drift field~$\overline{\xvec{y}}$ is tangential to~$\mathscr{C}$ at~$\Gamma$, no boundary conditions are required.
	\begin{thrm}\label{theorem:ControlLin}
		Fix any $m \in \mathbb{N}$, $\varepsilon > 0$, and target $\theta_1 \in \xHn{m}(\mathscr{C};\mathbb{R})$. There exists a control $g \in \xCinfty(\mathscr{C}\times[0,1];\mathbb{R})$ such that the solution $\theta \in \xCinfty(\mathscr{C}\times[0,1];\mathbb{R})$ to the linear problem
		\begin{equation*}\label{equation:BoussinesqLinearizedNonlocalized}
			\begin{gathered}
				\partial_t \theta + (\overline{\xvec{y}} \cdot \xnab) \theta = \mathbb{I}_{\omegaup}g, \quad \theta(\cdot, 0) = 0
			\end{gathered}
		\end{equation*}
		obeys $\|\theta(\cdot,1) - \theta_1\|_{m-1} < \varepsilon \|\theta_1\|_m$. Moreover, for a constant $C_{\varepsilon} > 0$ depending only on~$\varepsilon$, the control can be chosen such that $\|g\|_{ \xLtwo([0,1];\xHn{{m}}(\mathscr{C};\mathbb{R}))} < C_{\varepsilon} \|\theta_1\|_m$.
	\end{thrm}
	\begin{proof}
		By density and compactness arguments, we take $\widetilde{\theta}_1 \in \xCinfty(\mathscr{C};\mathbb{R})$ with $\|\theta_1 - \widetilde{\theta}_1\|_{m-1} < \varepsilon\|\theta_1\|_m$ and such that $\|\widetilde{\theta}_1\|_{m+1} < \widetilde{C}_{\varepsilon} \|\theta_1\|_m$ for some $\widetilde{C}_{\varepsilon} > 1$ depending only on $\varepsilon$. Then, we define $\widetilde{\theta}(\cdot, t) \coloneq \kappa(t) \widetilde{\theta}_1$, where $\kappa \in \xCinfty([0, 1]; [0,1])$ obeys $\operatorname{supp}(\kappa) \subset (0, 1]$ and $\kappa(1) = 1$. Moreover, we set $\widetilde{g} \coloneq \partial_t \widetilde{\theta} + (\overline{\xvec{y}} \cdot \xnab) \widetilde{\theta}$ and note that $\widetilde{\theta}$ solves the problem
		\begin{equation}\label{equation:tt}
			\partial_t \widetilde{\theta} + (\overline{\xvec{y}} \cdot \xnab) \widetilde{\theta} = \widetilde{g}, \quad \widetilde{\theta}(\cdot, 0) = 0
		\end{equation}
		and satisfies $\|\widetilde{\theta}(\cdot,1) - \theta_1\|_{m-1} < \varepsilon\|\theta_1\|_m$. 
		
		Now, we claim that by using instead of $\widetilde{g}$ the control
		\begin{equation}\label{equation:eqg}
			g(\xvec{x}, t) \coloneq \chi(x_2) \sum_{k=1}^K  \frac{1}{t_b^k-t_a^k} \mathbb{I}_{[t_a^k, t_b^k]}(t) \widetilde{g}\left( \xsym{\mathcal{Y}}\left(\xvec{x}, t, \frac{t-t_a^k}{t_b^k-t_a^k}\right), \frac{t-t_a^k}{t_b^k-t_a^k} \right),
		\end{equation}
		the solution to
		\begin{equation}\label{equation:BoussinesqLinearizedlocalizedA}
			\begin{gathered}
				\partial_t \theta + (\overline{\xvec{y}} \cdot \xnab) \theta = \mathbb{I}_{\omegaup}g, \quad \theta(\cdot, 0) = 0
			\end{gathered}
		\end{equation}
		will likewise obey $\|\theta(\cdot,1) - \theta_1\|_{m-1} < \varepsilon\|\theta_1\|_m$.
		To see this, we employ the well-known solution formulas for the involved transport problems. First, as $\xsym{\mathcal{Y}}$ is the flow that governs \eqref{equation:tt} and \eqref{equation:BoussinesqLinearizedlocalizedA}, it holds
		\begin{equation}\label{equation:sfmc}
			\widetilde{\theta}(\xvec{x}, 1) = \int_0^1 \widetilde{g}(\xsym{\mathcal{Y}}(\xvec{x}, 0, r), r) \, \xdx{r}, \quad \theta(\xvec{x}, 1) = \int_0^1 g(\xsym{\mathcal{Y}}(\xvec{x}, 0, s), s) \, \xdx{r},
		\end{equation}
		noting that $\operatorname{supp}(\chi) \subset \omegaup$, and recalling from \Cref{theorem:convection} that $\xsym{\mathcal{Y}}(\xvec{x}, 0, 1) = \xsym{\mathcal{Y}}(\xvec{x}, 1, 0) = \xvec{x}$ for $\xvec{x} \in \mathscr{C}$.
		Second, it can be shown as follows that both integrals in \eqref{equation:sfmc} have the same value:
		\begin{multline}\label{equation:locstr}
			\int_0^1 g(\xsym{\mathcal{Y}}(\xvec{x}, 0, s), s) \, \xdx{s} \\
			\begin{aligned}
				& = \sum_{k=1}^K  \int_0^1 \frac{\mathbb{I}_{[t_a^k, t_b^k]}(s)}{t_b^k-t_a^k}  \chi(\xsym{\mathcal{Y}}(\xvec{x}, 0, s)) \widetilde{g} \left(\xsym{\mathcal{Y}}\left(\xvec{x}, 0 , \frac{s-t_a^k}{t_b^k-t_a^k} \right), \frac{s-t_a^k}{t_b^k-t_a^k}\right) \, \xdx{s} \\
				& = \sum_{k=1}^K \int_0^1 \chi\left( \xsym{\mathcal{Y}}\left(\xvec{x}, 0, r (t_b^k-t_a^k) + t_a^k \right)\right) \widetilde{g} (\xsym{\mathcal{Y}}(\xvec{x}, 0 , r), r) \, \xdx{r} \\
				& = \int_0^1 \widetilde{g}(\xsym{\mathcal{Y}}(\xvec{x}, 0, r), r) \, \xdx{r},
			\end{aligned}
		\end{multline}
		where we used {\it P1-P3} from \Cref{theorem:convection}, the substitutions $r = (s-t_a^k)(t_b^k-t_a^k)^{-1}$ for $k \in \{1,\dots, K\}$, and the properties of~$\chi$ from \eqref{equation:chip}. In particular, regarding the last equality in \eqref{equation:locstr}, we note that $\chi$ generates the partition of unity $(\chi(\cdot + 3(i-1)l_{K}/4))_{i\in\{1,\dots,K\}}$, and one can employ {\it P3}, because $r (t_b^k-t_a^k) + t_a^k \in [t_a^k, t_b^k]$ for all $r \in [0,1]$; as $\overline{\xvec{y}}$ only depends on time, this yields for any fixed $\xvec{x} \in \mathscr{C}$ that
		\[
			\sum_{k=1}^K \chi\left( \xsym{\mathcal{Y}}\left(\xvec{x}, 0, r (t_b^k-t_a^k) + t_a^k \right)\right) = 1, \quad r \in [0,1].
		\]
		
		The bound for $\smash{\|g\|_{ \xLtwo([0,1];\xHn{{m}}(\mathscr{C};\mathbb{R}))}}$ can be concluded from \eqref{equation:eqg} and the definition of~$\widetilde{\theta}$, noting that $\overline{\xvec{y}}$ and its flow~$\xsym{\mathcal{Y}}$ are smooth and universally fixed. Also, at first it only holds $g \in \xLtwo([0,1];\xCinfty(\mathscr{C};\mathbb{R}))$, but we can approximate $g$ by a $\xCinfty(\mathscr{C}\times[0,1];\mathbb{R})$ version while maintaining $\|\theta(\cdot,1) - \theta_1\|_{m-1} < \varepsilon \|\theta_1\|_m$.
	\end{proof}
	
	\begin{rmrk}\label{remark:nm}
		In the proof of \Cref{theorem:ControlLin} with $m \geq 2$, if one approximates $\smash{\theta_1 \in \xHn{{m}}_{\operatorname{N}}}$ in the there-described way by $\widetilde{\theta}_1 \in \xHn{{m+1}}_{\operatorname{N}}$, then a direct calculation, using the properties of $\overline{\xvec{y}}$ and $\chi$ in \eqref{equation:eqg}, provides $\partial_{\xvec{n}} \mathbb{I}_{\omegaup} g (\cdot, t) |_{\Gamma} = 0$ for $t \in [0,1]$.
	\end{rmrk}
	
	Now, we steer the temperature in the nonlinear problem, while ensuring that the final vorticity is a small perturbation of the initial one; for the torus case, see~\cite{NersesyanRissel2024}. The employed return method argument is inspired by \cite{Coron96} and \cite{Nersesyan2021,NersesyanRissel2022}, and the idea is to view a scaled solution to \eqref{equation:Boussinesq} on a small time interval as a perturbation of an accordingly scaled return method trajectory, as described in \Cref{remark:return}.
	\begin{thrm}\label{theorem:HDS}
		Let $T > 0$, $\varepsilon > 0$, $(w_0, \theta_0, \theta_1) \in \xHn{{2}}_{0} \times \xHn{{3}}_{\operatorname{N}} \times \xHn{{3}}_{\operatorname{N}}$, and~forces $\varphi = \xwcurl{\xsym{\Phi}}$ and $\psi$ given by \Cref{theorem:main}. There exists $\delta > 0$ and $\eta \in \xCinfty(\mathscr{C}\times[0,\delta]; \mathbb{R})$ such that
		\begin{equation}\label{equation:cc}
			\|S_{\delta}\left(w_0, \theta_0, \varphi, \psi + \eta\right) - (w_0, \theta_1)\|_{\xHn{{1}} \times \xHn{{2}}} < \varepsilon.
		\end{equation}
		Moreover, the velocity $\xvec{u}_{\delta}$, associated to $w_0$, $\theta_0$, $\varphi$, $\psi$, and $\eta$ via \eqref{equation:vorticity}, obeys
		\[
			\int_{\mathscr{C}} \xvec{u}_{\delta}(\xvec{x}, t) \cdot \xsym{\mathscr{g}} \, \xdx{\xvec{x}} = [0, \overline{y}_{2,\delta}(t)]^{\top} = \overline{\xvec{y}}_{\delta}(t) \coloneq \delta^{-1} \overline{\xvec{y}}(\delta^{-1}t), \quad t \in [0, \delta].
		\]
	\end{thrm}
	\begin{proof}
		First, a family of auxiliary controls is constructed via \Cref{theorem:ControlLin}. Second, a controlled nonlinear trajectory is fixed. Third, asymptotic expansions are proposed for verifying~\eqref{equation:cc}. Finally, the remainder estimates are given.
		
		\paragraph{Step 1. Controlling a family of linear problems.}
		For~$\delta \in (0, 1)$, we denote by $\widetilde{\Theta}_{\delta}$ the unique solution to the uncontrolled transport equation with scaled temperature initial data
		\begin{equation}\label{equation:FreeBoussinesqLinearizedlocalized}
			\begin{gathered}
				\partial_t \widetilde{\Theta}_{\delta} + (\overline{\xvec{y}} \cdot \xnab) \widetilde{\Theta}_{\delta} = 0, \quad
				\widetilde{\Theta}_{\delta}(\cdot, 0) = \delta\theta_0, 
			\end{gathered}
		\end{equation}
		emphasizing that $\partial_{\xvec{n}} \widetilde{\Theta}_{\delta}(\cdot, t) |_{\Gamma} = 0$ for all $t \in [0, 1]$ due to $\theta_0 \in \xHn{{3}}_{\operatorname{N}}$. Since~$\xsym{\mathcal{Y}}$ is the flow of~$\overline{\xvec{y}}$ from \Cref{theorem:convection}, we have $\widetilde{\Theta}_{\delta}(\cdot, 1) = \delta\theta_0$.
		Now, for any $\delta \in (0, 1)$, we apply \Cref{theorem:ControlLin}, in the way described by \Cref{remark:nm}, with the target state $\delta(\theta_1 - \theta_0)$. This yields a family of smooth controls $(g_{\delta})_{\delta\in(0,1)}$, spatially supported in $\omegaup$, such that the solution to the problem
		\begin{equation}\label{equation:ctrllinprb}
			\begin{gathered}
				\partial_t \widehat{\Theta}_{\delta} + (\overline{\xvec{y}} \cdot \xnab) \widehat{\Theta}_{\delta} = \mathbb{I}_{\omegaup}g_{\delta}, \quad
				\widehat{\Theta}_{\delta}(\cdot, 0) = 0
			\end{gathered}
		\end{equation} 
		satisfies $\|\widehat{\Theta}_{\delta}(\cdot,1) - \delta(\theta_1 - \theta_0)\|_2 < \varepsilon \delta\|\theta_1 - \theta_0\|_3$ and $\partial_{\xvec{n}} \widehat{\Theta}_{\delta}(\cdot, t) |_{\Gamma} = 0$. Using the linearity of \eqref{equation:FreeBoussinesqLinearizedlocalized} and \eqref{equation:ctrllinprb}, it follows that $\smash{\overline{\vartheta}_{\delta} \coloneq \widetilde{\Theta}_{\delta} + \widehat{\Theta}_{\delta}}$
		obeys
		\begin{equation}\label{equation:StLinearizedlocalized}
			\begin{gathered}
				\partial_t \overline{\vartheta}_{\delta} + (\overline{\xvec{y}} \cdot \xnab) \overline{\vartheta}_{\delta} = \mathbb{I}_{\omegaup}g_{\delta}, \\
				\overline{\vartheta}_{\delta}(\cdot, 0) = \delta\theta_0, \quad \|\overline{\vartheta}_{\delta}(\cdot, 1) - \delta \theta_1\|_2 < \varepsilon \delta\|\theta_1 - \theta_0\|_3.
			\end{gathered}
		\end{equation}
		However, in the steps below, we will work with the following average-free version:
		\[
			\widetilde{\vartheta}_{\delta}(\xvec{x}, t) \coloneq \overline{\vartheta}_{\delta}(\xvec{x}, t) - \frac{\chi(x_2) \int_0^t \int_{\mathscr{C}} g_{\delta}(\xvec{z}, s) \, \xdx{\xvec{z}} \xdx{s}}{\int_{\mathscr{C}} \chi(\xvec{z}) \, \xdx{\xvec{z}}},
		\]
		which obeys -- by \Cref{remark:nm}, \eqref{equation:StLinearizedlocalized}, and $\theta_0, \theta_1 \in \xHn{{3}}_{\operatorname{N}}$ -- the initial, target, and boundary conditions
		\begin{equation}\label{equation:intag}
			\widetilde{\vartheta}_{\delta}(0) = \delta \theta_0, \quad \|\widetilde{\vartheta}_{\delta}(1) - \delta\theta_1\|_2 < \varepsilon\delta\|\theta_1 - \theta_0\|_3, \quad \partial_{\xvec{n}}\widetilde{\vartheta}_{\delta} |_{\Gamma} = 0.
		\end{equation}
		Moreover, together with the control
		\begin{equation}\label{equation:definition_eta_delta}
			\begin{aligned}
				\widetilde{\eta}_{\delta}(\xvec{x}, t) \coloneq g_{\delta}(\xvec{x}, t) 
				- \frac{\overline{y}_2 \chi'(x_2) \int_0^t \int_{\mathscr{C}} g_{\delta}(\xvec{z}, s) \, \xdx{\xvec{z}}  \xdx{s} + \chi(x_2) \int_{\mathscr{C}} g_{\delta}(\xvec{z}, t) \, \xdx{\xvec{z}}}{\int_{\mathscr{C}} \chi(\xvec{z}) \, \xdx{\xvec{z}}},
			\end{aligned}
		\end{equation}
		it holds
		\begin{equation}\label{equation:wtvt}
			\partial_t \widetilde{\vartheta}_{\delta} + (\overline{\xvec{y}} \cdot \xnab) \widetilde{\vartheta}_{\delta} = \mathbb{I}_{\omegaup}\widetilde{\eta}_{\delta}.
		\end{equation}
		Meanwhile, we define the corresponding function $\widetilde{v}_{\delta}$ as the solution to
		\begin{equation}\label{equation:wtvtvv}
			\partial_t \widetilde{v}_{\delta} + (\overline{\xvec{y}} \cdot \xnab) \widetilde{v}_{\delta} = \partial_1 \widetilde{\vartheta}_{\delta}, \quad \widetilde{v}_{\delta}(\cdot, 0) = w_0,
		\end{equation}
		observing that $\widetilde{v}_{\delta} |_{\Gamma} = 0$ due to \eqref{equation:intag} and the choice of $w_0$. Direct estimates based on the solution representations for the transport problems in \cref{equation:FreeBoussinesqLinearizedlocalized,equation:ctrllinprb,equation:wtvtvv,equation:wtvt,equation:StLinearizedlocalized} provide
		\begin{equation}\label{equation:tvartassymptassump}
			\begin{gathered}
				\sup_{t \in [0, 1]} \|\widetilde{v}_{\delta}(\cdot, t) - w_0\|_{2} + \sup_{t \in [0, 1]} \|\widetilde{\vartheta}_{\delta}(\cdot, t)\|_{3} \lesssim C_{\varepsilon}\delta,
			\end{gathered}
		\end{equation}
		where $C_{\varepsilon} > 1$ is the fixed constant from \Cref{theorem:ControlLin} depending only on $\varepsilon$.
		
		\paragraph{Step 2. Controls for the nonlinear problem.}
		Given $(\widetilde{\eta}_{\delta})_{\delta \in (0,1)}$ from \eqref{equation:definition_eta_delta}, we aim to relate, for small $\delta$, the final states of suitably controlled trajectories of the nonlinear problem \eqref{equation:vorticity} to those of the linear systems  with parameter~$\delta$ constituted by \eqref{equation:wtvt} and \eqref{equation:wtvtvv}. Hereto, in the next step (see \eqref{equation:ansatz}), we will view~$\xvec{u}_{\delta}$ on the short time interval $[0, \delta]$ as a perturbation of the leading order profile $\overline{\xvec{y}}_{\delta}(t) = \delta^{-1} \overline{\xvec{y}}(\delta^{-1}t)$. But for this to be feasible, the velocity~$\xvec{u}_{\delta}$ should have the vertical average $\smash{\smallint_{\mathscr{C}} \xvec{u}_{\delta}(\xvec{x}, t) \cdot \xsym{\mathscr{g}}(\xvec{x}) \, \xdx{\xvec{x}} = \overline{\xvec{y}}_{\delta}(t)}$ for all $t \in [0, \delta]$. To ensure this property, we can act on the velocity average through the temperature control. Namely, for each $\delta \in (0, 1)$, we denote (\cf~\Cref{remark:tc})
		\begin{equation}\label{equation:mc}
			\eta_{\delta} \coloneq \delta^{-2} \widetilde{\eta}_{\delta}(\cdot, \delta^{-1}\cdot) + \frac{\overline{y}_{2,\delta}''\chi - \tau \overline{y}_{2,\delta}'\Delta \chi + \overline{y}_{2,\delta}'(\widetilde{\xvec{u}}_{\delta} \cdot \xnab) \chi}{\int_{\mathscr{C}}\chi(x_2) \, \xdx{\xvec{x}}},
		\end{equation}
		where $\chi$ is from \eqref{equation:chi} and~$\widetilde{\xvec{u}}_{\delta}$ is the solution to
		\begin{gather*}
			\xwcurl{\widetilde{\xvec{u}}_{\delta}}(\cdot, t) = S^1_{t}(w_0, \theta_0,  \varphi,  \psi + \delta^{-2} \widetilde{\eta}_{\delta}(\cdot, \delta^{-1}\cdot) + \overline{y}_{2,\delta}''), \\  
			\xdiv{\widetilde{\xvec{u}}_{\delta}} = 0, \quad \widetilde{\xvec{u}}_{\delta}|_{\Gamma} \cdot \xvec{n} = 0, \quad
			\int_{\mathscr{C}} \widetilde{\xvec{u}}_{\delta}(\xvec{x}, t) \cdot \xsym{\mathscr{g}} \, \xdx{\xvec{x}} = \overline{\xvec{y}}_{\delta}(t),
		\end{gather*}
		Next, we fix the controlled trajectory
		\begin{equation}\label{equation:controlledtrajectory}
			(w_{\delta}, \theta_{\delta})(\cdot, t) \coloneq S_{t}(w_0, \theta_0,  \varphi,  \psi + \eta_{\delta}), \quad t \in [0, \delta]
		\end{equation}
		and denote by $\xvec{u}_{\delta}$ the associated velocity.
		Now, due to the above constructions, we can verify that
		\[
			\xwcurl{\xvec{u}_{\delta}} = w_{\delta}, \quad  \xdiv{\xvec{u}_{\delta}} = 0, \quad  \xvec{u}_{\delta}|_{\Gamma} \cdot \xvec{n} = 0, \quad \int_{\mathscr{C}} \xvec{u}_{\delta}(\xvec{x}, t) \cdot \xsym{\mathscr{g}} \, \xdx{\xvec{x}} = \overline{\xvec{y}}_{\delta}(t)
		\]
		and
		\begin{equation}\label{equation:uequtilde}
			\xvec{u}_{\delta} = \widetilde{\xvec{u}}_{\delta}.
		\end{equation}
		Indeed, each $\widetilde{\eta}_{\delta}(\cdot, t)$ is average-free for almost all $t \in (0,1)$ because of \eqref{equation:definition_eta_delta}, and the second term in the right-hand sides of \eqref{equation:mc} ensures that (\cf~\cref{equation:restr,equation:vorticity})
		\begin{gather*}
			\theta_{\delta}(\cdot, t) = S^2_{t}(w_0, \theta_0,  \varphi,  \psi + \delta^{-2} \widetilde{\eta}_{\delta}(\cdot, \delta^{-1}\cdot)) + \frac{\overline{y}_{2,\delta}'\chi}{\int_{\mathscr{C}}\chi(x_2) \, \xdx{\xvec{x}}}, \\
			\int_{\mathscr{C}} \xvec{u}_{\delta}(\xvec{x}, t) \cdot \xsym{\mathscr{g}} \, \xdx{\xvec{x}} = \int_0^t\int_{\mathscr{C}} \theta_{\delta}(\xvec{x}, s) \, \xdx{\xvec{x}} \xdx{s} = \overline{y}_{2,\delta}(t).
		\end{gather*}
		In order to confirm \eqref{equation:uequtilde}, one notes that $\overline{\xvec{y}}_{\delta}'(t) - \overline{\xvec{y}}_{\delta}'(t)\chi/\smallint_{\mathscr{C}} \chi(\xvec{z})\,\xdx{\xvec{z}}$ is curl-free and average-free for $t\in[0,\delta]$, implying $\overline{\xvec{y}}_{\delta}' - \overline{\xvec{y}}_{\delta}'\chi/\smallint_{\mathscr{C}} \chi(\xvec{z})\,\xdx{\xvec{z}} = \nabla q$ with smooth~$q$. Hence,~$\xvec{u}_{\delta}$ and~$\widetilde{\xvec{u}}_{\delta}$ solve the same velocity equation with identical data. See also \Cref{remark:return}.
		\begin{rmrk}\label{remark:tc}
			The definition in \eqref{equation:mc} has the effect, that if $\rho$ solves $\partial_t\rho - \tau\Delta\rho + (\widetilde{\xvec{u}}_{\delta} \cdot \xnab) \rho = \delta^{-2} \widetilde{\eta}_{\delta}(\cdot, \delta^{-1}\cdot)$, then $\widetilde{\rho} \coloneq \rho + \chi \overline{y}_{2,\delta}'$ obeys $\partial_t\widetilde{\rho} - \tau\Delta\widetilde{\rho} + (\widetilde{\xvec{u}}_{\delta} \cdot \xnab) \widetilde{\rho} = \eta_{\delta}$.
		\end{rmrk}

		\paragraph{Step 3. Asymptotic expansions.} Given the trajectory $(w_{\delta}, \theta_{\delta})$ defined in \eqref{equation:controlledtrajectory}, we make on the time interval $[0, \delta]$ an ansatz of the form
		\begin{equation}\label{equation:ansatz}
			\begin{gathered}
				w_{\delta} = z_{\delta} + q_{\delta}, \quad \xvec{u}_{\delta} = \overline{\xvec{y}}_{\delta} + \xvec{Z}_{\delta} + \xvec{Q}_{\delta}, \quad	\theta_{\delta} = \vartheta_{\delta} + \frac{\overline{y}_{2,\delta}'\chi}{\int_{\mathscr{C}}\chi(x_2) \, \xdx{\xvec{x}}} + r_{\delta},
			\end{gathered}
		\end{equation}
		where 
		\[
			z_{\delta}(\cdot,t) \coloneq \widetilde{v}_{\delta}(\cdot, \delta^{-1}t), \quad \vartheta_{\delta}(\cdot,t) \coloneq \delta^{-1}\widetilde{\vartheta}_{\delta}(\cdot,\delta^{-1}t)
		\]
		are determined through \eqref{equation:wtvt} and \eqref{equation:wtvtvv}, while the vector fields~$\xvec{Z}_{\delta}$ and~$\xvec{Q}_{\delta}$ are the unique functions satisfying
		\begin{equation*}\label{equation:dcZQ}
			\begin{gathered}
				\xwcurl{\xvec{Z}_{\delta}} = z_{\delta}, \quad \xdiv{\xvec{Z}_{\delta}} = 0, \quad \xvec{Z}_{\delta}|_{\Gamma} \cdot \xvec{n} = 0, \quad  \int_{\mathscr{C}} \xvec{Z}_{\delta}(\xvec{x}, t) \cdot \xsym{\mathscr{g}} \, \xdx{\xvec{x}} = 0,\\
				\xwcurl{\xvec{Q}_{\delta}} = q_{\delta}, \quad \xdiv{\xvec{Q}_{\delta}} = 0, \quad \xvec{Q}_{\delta}|_{\Gamma} \cdot \xvec{n} = 0, \quad \int_{\mathscr{C}} \xvec{Q}_{\delta}(\xvec{x}, t) \cdot \xsym{\mathscr{g}} \, \xdx{\xvec{x}} = 0.
			\end{gathered}
		\end{equation*}
		It should be emphasized that, due to the boundary data of~$\theta_0$ and~$\theta_1$, and the Dirichlet condition for $w_0$, one has $z_{\delta} |_{\Gamma} = 0$ and $\partial_{\xvec{n}}\vartheta_{\delta} |_{\Gamma} = 0$; see also~\Cref{remark:nm}, \cref{equation:intag,equation:wtvt,equation:wtvtvv}.
		In view of \eqref{equation:intag}, \eqref{equation:ansatz}, and the compact support of~$\overline{\xvec{y}}$, it remains to show
		\begin{equation}\label{equation:hdsgoal}
			\|q_{\delta}(\cdot, \delta)\|_{1} + \|r_{\delta}(\cdot, \delta)\|_{2} \longrightarrow 0 \mbox{ as } \delta \longrightarrow 0.
		\end{equation}
		The limit \eqref{equation:hdsgoal} can be verified via energy estimates for the equations satisfied by~$q_{\delta}$ and~$r_{\delta}$, namely
		\begin{equation}\label{equation:remainders}
			\begin{gathered}
				\partial_t q_{\delta} - \nu \Delta q_{\delta} + \left((\overline{\xvec{y}}_{\delta} + \xvec{Z}_{\delta} + \xvec{Q}_{\delta}) \cdot \xnab\right) q_{\delta} + (\xvec{Q}_{\delta} \cdot \xnab) z_{\delta} = F_{\delta} + \partial_1 r_{\delta}, \\
				\partial_t r_{\delta} - \tau\Delta r_{\delta} + \left((\overline{\xvec{y}}_{\delta} + \xvec{Z}_{\delta} + \xvec{Q}_{\delta}) \cdot \xnab\right) r_{\delta} + (\xvec{Q}_{\delta} \cdot \xnab) \vartheta_{\delta} = G_{\delta},\\
				q_{\delta}(\cdot, 0) = 0, \quad r_{\delta}(\cdot, 0) = 0, \quad q_{\delta} |_{\Gamma} = 0, \quad \partial_{\xvec{n}}r_{\delta} |_{\Gamma} = 0,
			\end{gathered}
		\end{equation}
		where $F_{\delta} \coloneq \varphi - (\xvec{Z}_{\delta} \cdot \xnab) z_{\delta} + \nu \Delta z_{\delta}$ and  $G_{\delta} \coloneq \psi - (\xvec{Z}_{\delta} \cdot \xnab) \vartheta_{\delta} + \tau \Delta \vartheta_{\delta}$.
		To obtain \eqref{equation:remainders}, one inserts \cref{equation:ansatz,equation:controlledtrajectory} into \eqref{equation:vorticity}, and further utilizes that $\partial_1 \chi \equiv 0$ by \eqref{equation:chi}, as well as the expression of~$\eta_{\delta}$ from \eqref{equation:mc} in conjunction with \eqref{equation:uequtilde}.
		\paragraph{Step 4. Estimates.}
		Let $t \in [0, \delta]$, and recall the Sobolev embeddings $\xHone(\mathscr{C}; \mathbb{R}) \subset \xLfour(\mathscr{C}; \mathbb{R})$ and $\xHtwo(\mathscr{C}; \mathbb{R}) \subset \xLinfty(\mathscr{C}; \mathbb{R})$. We multiply in \eqref{equation:remainders} respectively with $(-\Delta)^iq_{\delta}$ and $(-\Delta)^ir_{\delta}$ for $i = 0, 1$, and then use integration by parts. By employing the boundary values of the involved functions,~\eqref{equation:dce}, and \eqref{equation:LaplH2}, we find
		\begin{multline}\label{equation:e1}
			\| q_{\delta} (\cdot, t)\|_1^2 + \| r_{\delta} (\cdot, t)\|_1^2 + (1 - \ell)\int_0^t \left(\| q_{\delta}(\cdot, s) \|_2^2 + \| r_{\delta}(\cdot, s) \|_2^2\right) \, \xdx{s} \\
			\begin{aligned}
				& \lesssim \int_0^t  \ell^{-1} \left(\|F_{\delta}(\cdot, s)\|^2 + \|G_{\delta}(\cdot, s)\|^2 + \|q_{\delta}(\cdot, s)\|_1^2 + \|r_{\delta}(\cdot, s)\|_1^2\right) \, \xdx{s} \\
				& \quad + \int_0^t \ell^{-1}\left(\|q_{\delta}(\cdot, s)\|_1^4  + \|r_{\delta}(\cdot, s)\|_1^4 + \|\vartheta_{\delta}(\cdot, s)\|_1^4\right)  \, \xdx{s} \\
				& \quad + \int_0^t \left(\|z_{\delta}(\cdot, s)\|_2 + |\overline{\xvec{y}}_{\delta}(\cdot, s)|\right)\left(\|q_{\delta}(\cdot, s)\|_1^2 + \|r_{\delta}(\cdot, s)\|_1^2\right)  \, \xdx{s},
			\end{aligned}
		\end{multline}
		where $\ell \in (0,1)$ is fixed independently of $\delta$. Next, after acting in the equation for~$r_{\delta}$ in \eqref{equation:remainders} with $\partial_i$, where~$i \in \{1,2\}$, it follows that
		\begin{multline*}
			\partial_t \partial_i r_{\delta} - \tau \Delta \partial_i r_{\delta}  + \left((\overline{\xvec{y}}_{\delta} + \xvec{Z}_{\delta} + \xvec{Q}_{\delta}) \cdot \xnab\right) \partial_i r_{\delta} + (\xvec{Q}_{\delta} \cdot \xnab) \partial_i\vartheta_{\delta} \\ = \partial_i G_{\delta} - \left((\partial_i\xvec{Z}_{\delta} + \partial_i\xvec{Q}_{\delta}) \cdot \xnab\right) r_{\delta} - (\partial_i\xvec{Q}_{\delta} \cdot \xnab) \vartheta_{\delta}.
		\end{multline*}
		Multiplying with $-\Delta \partial_i r_{\delta}$, integrating by parts, and using the known boundary values of the involved functions, we obtain together with \eqref{equation:e1} that
		\begin{multline*}
			\| q_{\delta} (\cdot, t)\|_1^2 + \| r_{\delta} (\cdot, t)\|_2^2 + (1 - \ell)\int_0^t \left(\| q_{\delta}(\cdot, s) \|_2^2 + \| r_{\delta}(\cdot, s) \|_3^2 \right) \, \xdx{s} \\
			\begin{aligned}
				& \lesssim \int_0^t  \ell^{-1} \left(\|F_{\delta}(\cdot, s)\|^2 + \|G_{\delta}(\cdot, s)\|_1^2 + \|q_{\delta}(\cdot, s)\|_1^2 + \|r_{\delta}(\cdot, s)\|_2^2\right) \, \xdx{s} \\
				& \quad + \int_0^t \ell^{-1}\left(\|q_{\delta}(\cdot, s)\|_1^4  + \|r_{\delta}(\cdot, s)\|_2^4 + \|\vartheta_{\delta}(\cdot, s)\|_2^4\right)  \, \xdx{s} \\
				& \quad + \int_0^t \left(\|z_{\delta}(\cdot, s)\|_2 + |\overline{\xvec{y}}_{\delta}(\cdot, s)|\right)\left(\|q_{\delta}(\cdot, s)\|_1^2 + \|r_{\delta}(\cdot, s)\|_2^2\right)  \, \xdx{s},
			\end{aligned}
		\end{multline*}
		with fixed $\ell \in (0,1)$ independent of $\delta$. Moreover, the forces $\varphi = \xwcurl{\xsym{\Phi}}$ and $\psi$ are fixed in \Cref{theorem:main}, and the norms $\|z_{\delta}(\cdot, t)\|_{2}$ and $\|\vartheta_{\delta}(\cdot, t)\|_{3}$ are for $t \in [0,\delta]$ bounded by an absolute constant. Thus, by noting that
		\begin{gather*}
			\|(\xvec{Z}_{\delta} \cdot \xnab) z_{\delta}(\cdot, t)\| \leq \|z_{\delta}(\cdot, t)\|_{1}^2, \quad \|\Delta z_{\delta}(\cdot, t)\| \leq \|z_{\delta}(\cdot, t)\|_{2},\\
			\|(\xvec{Z}_{\delta} \cdot \xnab) \vartheta_{\delta}(\cdot, t)\|_{1} \leq \|z_{\delta}(\cdot, t)\|_{1}^2 + \|\vartheta_{\delta}(\cdot, t)\|_{2}^2, \quad \|\Delta \vartheta_{\delta}(\cdot, t)\|_{1} \leq \|\vartheta_{\delta}(\cdot, t)\|_{3},
		\end{gather*}
		we can deduce (using the substitution $s = \delta \sigma$)
		\begin{gather*}
			\lim\limits_{\delta \to 0} \int_0^{\delta} \left(\|F_{\delta}(\cdot, s)\| + \|G_{\delta}(\cdot, s)\|_{1}\right) \, \xdx{s} = 0,\\
			\int_0^t \left(\|z_{\delta}(\cdot, s)\|_2 + |\overline{\xvec{y}}_{\delta}(\cdot, s)|\right) \, \xdx{s} = \int_0^1 \left(\delta\|\widetilde{v}_{\delta}(\cdot, \sigma)\|_2 + |\overline{\xvec{y}}(\sigma)|\right) \, \xdx{\sigma} \leq M_{\delta},
		\end{gather*}
		where $(M_{\delta})_{\delta \in (0,1)}$ is a bounded family. In particular, thanks to \eqref{equation:tvartassymptassump}, one has $\lim_{\delta\to 0}M_{\delta} = \sup_{t \in [0,1]} |\overline{\xvec{y}}(t)|$. In conclusion, using Gr\"onwall's inequality, 
		\begin{equation*}
			\|q_{\delta}(\cdot, t)\|^2_1 + \|r_{\delta}(\cdot, t)\|^2_2
			\lesssim c_{\delta} + \int_0^t \| \left(q_{\delta}(\cdot, s) \|_1^4 + \| r_{\delta}(\cdot, s) \|_2^4\right) \, \xdx{s}, 
		\end{equation*}
		where $\lim_{\delta \to 0} c_{\delta} = 0$. Like in the proof of \Cref{theorem:LCST}, comparing with $f'/f^2 \leq C$ for an absolute constant $C > 0$, the limit \eqref{equation:hdsgoal} follows.
	\end{proof}

	\subsection{Conclusion}\label{subsection:conclusion}
	To conclude \Cref{theorem:main}, it suffices, in view of the estimate \eqref{equation:dce}, to obtain for any $\varepsilon > 0$ a control $\eta \in \xCinfty(\mathscr{C}\times[0,T])$ with $\operatorname{supp}(\eta) \subset \omegaup\times(0,T)$ such that the corresponding solution to \eqref{equation:Boussinesq} obeys
	\begin{equation*}\label{equation:controllabilityproperty_vorticity_average}
		\|(\xwcurl{\xvec{u}})(\cdot, T) - w_T\|_{1} + \|\theta(\cdot, T) - \theta_T\|_{2} + |\int_{\mathscr{C}} \xvec{u}(\xvec{x}, T) \cdot \xsym{\mathscr{g}} \, \xdx{\xvec{x}}|  < \varepsilon,
	\end{equation*}
	where $w_T \coloneq \xwcurl{\xvec{u}}_T$. 
	
	In what remains, we specify a control~$\eta$, and the corresponding controlled trajectory, on respective time intervals $[0, T_1]$, $[T_1, T_2]$, $[T_2, T_3]$, and $[T_3, T]$, with $0 < T_1 < T_2 < T_3 < T$.

	\paragraph{Step 1. Smoothing.} Take $T_1 \in (0, T)$ such that the corresponding solution $(w, \theta)$ to~\eqref{equation:vorticity} with \enquote{zero control}, \ie, $\eta = 0$, satisfies 
	\[
		(w_{T_1}, \theta_{T_1}) \coloneq (w(\cdot, T_1), \theta(\cdot, T_1)) \in \xHn{{2}}_{0} \times \xHn{{3}}_{\operatorname{N}}.
	\]
	As mentioned in Remarks~\Rref{remark:r0} and~\Rref{remark:r}, the assumptions on $(\xsym{\Phi}, \psi)$ in \Cref{theorem:main} allow us to access this parabolic smoothing for $(\xvec{u}, \theta)$ up to the regularity
	\[
		\smash{\xLinfty((0,T);\xHtwo(\mathscr{C};\mathbb{R}^{2+1}))\cap\xLtwo((0,T);\xHn{3}(\mathscr{C};\mathbb{R}^{2+1}))},
	\]
	with the details being provided, \eg, by \cite[Lemma 2.1]{Chaves-SilvaEtal2023} without external forces; if one would consider smooth external forces, for example, $(\xsym{\Phi}, \psi) = (\xsym{0}, 0)$, one would obtain solutions that are smooth for $t > 0$. See also \cite[Remark 3.2]{Temam1982} for such an argument taking into account external forces, but for the case of the no-slip boundary condition. 
	
	Owing to the global wellposedness of \eqref{equation:Boussinesq}, we take~$T_1 \in (0, T)$ so close to~$T$ that every uncontrolled trajectory of~\eqref{equation:vorticity} starting at~$t = \widetilde{T}$ inside the~$2\varepsilon/3$-ball with center $(w_T, \theta_T)$ in $\xHn{{1}}(\mathscr{C}; \mathbb{R})\times\xHn{2}(\mathscr{C}; \mathbb{R})$ will not cross the boundary of the~$\varepsilon$-ball with center $(w_T, \theta_T)$ in $\xHn{{1}}(\mathscr{C}; \mathbb{R})\times\xHn{2}(\mathscr{C}; \mathbb{R})$ until~$t = \widetilde{T} + (T-T_1)$.
	\paragraph{Step 2. Reaching a special temperature state.} Let $\xi \in \xCinfty(\mathscr{C}; \mathbb{R})$ be any average-free profile with $\partial_1\xi |_{\Gamma} = \partial_{111}\xi|_{\Gamma} = 0$ and so that
	\[
		\|(w_{T_1} - \partial_1\xi) - w_T\|_{1} < \varepsilon/3.
	\]
	Such a choice is possible due to the boundary conditions satisfied by~$w(\cdot, T_1)$ and~$w_T$, and by the density of $\xCinfty_0(\mathscr{C}; \mathbb{R})$ in
	$\xHn{{1}}_{0}$, where $\xCinfty_0(\mathscr{C}; \mathbb{R})$ denotes the smooth function with compact support in~$\mathscr{C}$. Indeed, by density, we first we take $h \in \xCinfty_0(\mathscr{C}; \mathbb{R})$ with $\|(w_{T_1} - h) - w_T\|_{1} < \varepsilon/3$. Next, we integrate $h$ in the $x_1$-direction and subsequently correct the average; namely
	\begin{equation*}\label{equation:xiconstr}
		\begin{gathered}
			\widehat{\xi}(x_1, x_2) \coloneq \int_{-1}^{x_1} h(s, x_2) \, \xdx{s}, \quad
			\xi \coloneq \widehat{\xi} - \int_{\mathscr{C}} \widehat{\xi}(\xvec{x}) \, \xdx{\xvec{x}},
		\end{gathered}
	\end{equation*}
	emphasizing that $\partial_1 \xi |_\Gamma = \partial_{111} \xi |_\Gamma = 0$, and that~$\xi$ has zero average.
	Now, using \Cref{theorem:HDS} and the smoothing property, we fix a small $\widetilde{\delta} \in (0, T-T_1)$ and a control $\eta \in \xCinfty(\mathscr{C}\times[0,T])$ such that the corresponding solution $(\xvec{u}, w, \theta)$ to~\eqref{equation:vorticity} starting from $t = T_1$ with data $(w_{T_1}, \theta_{T_1})$ satisfies at the time $T_2 \coloneq T_1 + \widetilde{\delta}$ the conditions
	\begin{equation*}\label{equation:rr}
		\begin{gathered}
			\|w_{T_2} - w_{T_1}\|_{1} < \beta, \quad \|\theta_{T_2} - \gamma^{-1}\xi\|_2 < \kappa, \quad \int_{\mathscr{C}} \xvec{u}_{T_2} \cdot \xsym{\mathscr{g}} \, \xdx{\xvec{x}} = 0,\\
			(\xvec{u}_{T_2}, w_{T_2}, \theta_{T_2}) \coloneq (\xvec{u}, w, \theta)(\cdot, T_2)) \in \xHn{{2}}_{0} \times \xHn{{3}}_{\operatorname{N}},
		\end{gathered}
	\end{equation*}
	where $\beta > 0$, $\gamma > 0$, and $\kappa > 0$ are selected so small that one has \eqref{equation:initialdatacontrolresult} with~$\varepsilon/2$ when applying \Cref{theorem:LCST} with initial data $(w_{T_2}, \theta_{T_2})$ and $\delta = \gamma$ and $T_2 + \gamma < T$.
	\paragraph{Step 3. Controlling the vorticity.}
	\Cref{theorem:LCST} provides $T_3 \in (T_2, T)$, \eg, $T_3 = T_2 + \gamma$, such that without using any control, the solution $(\xvec{u}, w, \theta)$ to~\eqref{equation:vorticity} on the interval $[T_2, T_3]$ issued at $t = T_2$ from $(w_{T_2}, \theta_{T_2})$, as obtained in the previous step, satisfies
	\[
		\|w(\cdot, T_3) - w_T\|_{1} < \varepsilon/2, \quad \int_{\mathscr{C}} \xvec{u}(\xvec{x}, T_3) \cdot \xsym{\mathscr{g}} \, \xdx{\xvec{x}} = 0.
	\] 

	\paragraph{Step 4. Correcting the temperature.}
	Finally, another application of \Cref{theorem:HDS} provides a time $T_4 \in (T_3, T)$ such that the solution $(w, \theta)$ to~\eqref{equation:Boussinesq} on the interval $[T_3, T_4]$ with initial data $w(\cdot, T_3)$ and $\theta(\cdot, T_3)$, as fixed above, obeys
	\[
		\|w(\cdot, T_4) - w_T\|_{1} + \|\theta(\cdot, T_4) - \theta_T\|_{2} < 2\varepsilon/3.
	\] 
	As $T_1$ was at the beginning chosen sufficiently close to $T$, the argument is complete.

	\bibliographystyle{alpha}
	\bibliography{Boussinesq}

\begin{bibdiv}
\begin{biblist}

\bib{AbergelTemam1990}{article}{
      author={Abergel, F.},
      author={Temam, R.},
       title={On some control problems in fluid mechanics},
        date={1990},
        ISSN={0935-4964},
     journal={Theor. Comput. Fluid Dyn.},
      volume={1},
      number={6},
       pages={303\ndash 325},
}

\bib{BoulvardGaoNersesyan2023}{article}{
      author={Boulvard, Pierre-Marie},
      author={Gao, Peng},
      author={Nersesyan, Vahagn},
       title={Controllability and ergodicity of three dimensional primitive
  equations driven by a finite-dimensional force},
        date={2023},
        ISSN={0003-9527,1432-0673},
     journal={Arch. Ration. Mech. Anal.},
      volume={247},
      number={1},
       pages={Paper No. 2, 49},
         url={https://doi.org/10.1007/s00205-022-01835-8},
}

\bib{Carreno2012}{article}{
      author={Carre\~{n}o, Nicol\'{a}s},
       title={Local controllability of the {$N$}-dimensional {B}oussinesq
  system with {$N-1$} scalar controls in an arbitrary control domain},
        date={2012},
        ISSN={2156-8472,2156-8499},
     journal={Math. Control Relat. Fields},
      volume={2},
      number={4},
       pages={361\ndash 382},
         url={https://doi.org/10.3934/mcrf.2012.2.361},
}

\bib{Chaves-SilvaEtal2023}{article}{
      author={Chaves-Silva, F.~W.},
      author={Fern\'{a}ndez-Cara, E.},
      author={Le~Balc'h, K.},
      author={Machado, J. L.~F.},
      author={Souza, D.~A.},
       title={Global controllability of the {B}oussinesq system with
  {N}avier-slip-with-friction and {R}obin boundary conditions},
        date={2023},
        ISSN={0363-0129,1095-7138},
     journal={SIAM J. Control Optim.},
      volume={61},
      number={2},
       pages={484\ndash 510},
         url={https://doi.org/10.1137/21M1425566},
}

\bib{Coron96}{article}{
      author={Coron, Jean-Michel},
       title={On the controllability of the {$2$}-{D} incompressible
  {N}avier-{S}tokes equations with the {N}avier slip boundary conditions},
        date={1995/96},
        ISSN={1292-8119},
     journal={ESAIM Contr\^{o}le Optim. Calc. Var.},
      volume={1},
       pages={35\ndash 75},
         url={https://doi.org/10.1051/cocv:1996102},
}

\bib{Coron2007}{book}{
      author={Coron, Jean-Michel},
       title={Control and nonlinearity},
      series={Mathematical Surveys and Monographs},
   publisher={American Mathematical Society, Providence, RI},
        date={2007},
      volume={136},
        ISBN={978-0-8218-3668-2; 0-8218-3668-4},
}

\bib{CoronLissy2014}{article}{
      author={Coron, Jean-Michel},
      author={Lissy, Pierre},
       title={Local null controllability of the three-dimensional
  {N}avier-{S}tokes system with a distributed control having two vanishing
  components},
        date={2014},
        ISSN={0020-9910,1432-1297},
     journal={Invent. Math.},
      volume={198},
      number={3},
       pages={833\ndash 880},
         url={https://doi.org/10.1007/s00222-014-0512-5},
}

\bib{CoronMarbachSueur2020}{article}{
      author={Coron, Jean-Michel},
      author={Marbach, Fr\'{e}d\'{e}ric},
      author={Sueur, Franck},
       title={Small-time global exact controllability of the {N}avier-{S}tokes
  equation with {N}avier slip-with-friction boundary conditions},
        date={2020},
        ISSN={1435-9855},
     journal={J. Eur. Math. Soc. (JEMS)},
      volume={22},
      number={5},
       pages={1625\ndash 1673},
}

\bib{CoronMarbachSueurZhang2019}{article}{
      author={Coron, Jean-Michel},
      author={Marbach, Fr\'{e}d\'{e}ric},
      author={Sueur, Franck},
      author={Zhang, Ping},
       title={Controllability of the {N}avier-{S}tokes equation in a rectangle
  with a little help of a distributed phantom force},
        date={2019},
        ISSN={2524-5317},
     journal={Ann. PDE},
      volume={5},
      number={2},
       pages={Art. 17},
}

\bib{Fernandez-CaraGuerreroImanuvilovPuel2006}{article}{
      author={Fern\'{a}ndez-Cara, Enrique},
      author={Guerrero, Sergio},
      author={Imanuvilov, Oleg~Yu.},
      author={Puel, Jean-Pierre},
       title={Some controllability results for the {$N$}-dimensional
  {N}avier-{S}tokes and {B}oussinesq systems with {$N-1$} scalar controls},
        date={2006},
        ISSN={0363-0129,1095-7138},
     journal={SIAM J. Control Optim.},
      volume={45},
      number={1},
       pages={146\ndash 173},
         url={https://doi.org/10.1137/04061965X},
}

\bib{FernandezCaraSantosSouza2016}{article}{
      author={Fern\'{a}ndez-Cara, Enrique},
      author={Santos, Maur\'{\i}cio~C.},
      author={Souza, Diego~A.},
       title={Boundary controllability of incompressible {E}uler fluids with
  {B}oussinesq heat effects},
        date={2016},
        ISSN={0932-4194,1435-568X},
     journal={Math. Control Signals Systems},
      volume={28},
      number={1},
       pages={Art. 7, 28},
         url={https://doi.org/10.1007/s00498-015-0158-x},
}

\bib{FoiasManleyTemam1987}{article}{
      author={Foias, C.},
      author={Manley, O.},
      author={Temam, R.},
       title={Attractors for the {B{\'e}nard} problem: {Existence} and physical
  bounds on their fractal dimension},
    language={English},
        date={1987},
        ISSN={0362-546X},
     journal={Nonlinear Anal., Theory Methods Appl.},
      volume={11},
       pages={939\ndash 967},
}

\bib{FursikovImanuvilov1998}{article}{
      author={Fursikov, A.~V.},
      author={Imanuvilov, O.~Yu.},
       title={Local exact boundary controllability of the {B}oussinesq
  equation},
        date={1998},
        ISSN={0363-0129,1095-7138},
     journal={SIAM J. Control Optim.},
      volume={36},
      number={2},
       pages={391\ndash 421},
         url={https://doi.org/10.1137/S0363012996296796},
}

\bib{Getling1998}{book}{
      author={Getling, A.~V.},
       title={Rayleigh-{B{\'e}nard} convection: {Structures} and dynamics},
      series={Adv. Ser. Nonlinear Dyn.},
   publisher={Singapore: World Scientific},
        date={1998},
      volume={11},
        ISBN={981-02-2657-8},
}

\bib{Guerrero2006b}{article}{
      author={Guerrero, S.},
       title={Local exact controllability to the trajectories of the
  {B}oussinesq system},
        date={2006},
        ISSN={0294-1449,1873-1430},
     journal={Ann. Inst. H. Poincar\'{e} C Anal. Non Lin\'{e}aire},
      volume={23},
      number={1},
       pages={29\ndash 61},
         url={https://doi.org/10.1016/j.anihpc.2005.01.002},
}

\bib{GuerreroMontoya2018}{article}{
      author={Guerrero, Sergio},
      author={Montoya, Cristhian},
       title={Local null controllability of the {{\(N\)}}-dimensional
  {N}avier-{S}tokes system with nonlinear {N}avier-{s}lip boundary conditions
  and {{\(N-1\)}} scalar controls},
    language={English},
        date={2018},
        ISSN={0021-7824},
     journal={J. Math. Pures Appl. (9)},
      volume={113},
       pages={37\ndash 69},
}

\bib{LiaoSueurZhang2022}{article}{
      author={Liao, Jiajiang},
      author={Sueur, Franck},
      author={Zhang, Ping},
       title={Smooth controllability of the {N}avier-{S}tokes equation with
  {N}avier conditions: application to {L}agrangian controllability},
        date={2022},
        ISSN={0003-9527},
     journal={Arch. Ration. Mech. Anal.},
      volume={243},
      number={2},
       pages={869\ndash 941},
}

\bib{Lions1996}{book}{
      author={Lions, Pierre-Louis},
       title={Mathematical topics in fluid mechanics. {V}ol. 1},
      series={Oxford Lecture Series in Mathematics and its Applications},
   publisher={The Clarendon Press, Oxford University Press, New York},
        date={1996},
      volume={3},
        ISBN={0-19-851487-5},
        note={Incompressible models, Oxford Science Publications},
}

\bib{Montoya2020}{article}{
      author={Montoya, Cristhian},
       title={Remarks on local controllability for the {B}oussinesq system with
  {N}avier boundary condition},
        date={2020},
        ISSN={1631-073X,1778-3569},
     journal={C. R. Math. Acad. Sci. Paris},
      volume={358},
      number={2},
       pages={169\ndash 175},
         url={https://doi.org/10.5802/crmath.29},
}

\bib{Navier1823}{article}{
      author={Navier, C.-L.-M.-H.},
       title={M{\'e}moire sur les lois du mouvement des fluides},
        date={1823},
     journal={M{\'e}moires de l’Acad{\'e}mie Royale des Sciences de
  l’Institut de France},
      volume={6},
      number={1823},
       pages={389\ndash 440},
}

\bib{Nersesyan2015}{article}{
      author={Nersesyan, Vahagn},
       title={Approximate controllability of {L}agrangian trajectories of the
  3{D} {N}avier-{S}tokes system by a finite-dimensional force},
        date={2015},
        ISSN={0951-7715,1361-6544},
     journal={Nonlinearity},
      volume={28},
      number={3},
       pages={825\ndash 848},
         url={https://doi.org/10.1088/0951-7715/28/3/825},
}

\bib{Nersesyan2021}{article}{
      author={Nersesyan, Vahagn},
       title={A proof of approximate controllability of the 3{D}
  {N}avier-{S}tokes system via a linear test},
        date={2021},
        ISSN={0363-0129},
     journal={SIAM J. Control Optim.},
      volume={59},
      number={4},
       pages={2411\ndash 2427},
         url={https://doi.org/10.1137/20M1368689},
}

\bib{NersesyanRissel2022}{article}{
      author={Nersesyan, Vahagn},
      author={Rissel, Manuel},
       title={Localized and degenerate controls for the incompressible
  navier--stokes system},
        date={2022},
     journal={arXiv preprint},
      eprint={https://doi.org/10.48550/arXiv.2212.01221},
}

\bib{NersesyanRissel2024}{article}{
      author={Nersesyan, Vahagn},
      author={Rissel, Manuel},
       title={Global controllability of boussinesq flows by using only a
  temperature control},
        date={2024},
     journal={arXiv preprint},
      eprint={https://doi.org/10.48550/arXiv:2404.09903},
}

\bib{Temam1982}{article}{
      author={Temam, R.},
       title={Behaviour at time {$t=0$} of the solutions of semilinear
  evolution equations},
        date={1982},
        ISSN={0022-0396},
     journal={J. Differential Equations},
      volume={43},
      number={1},
       pages={73\ndash 92},
}

\bib{Temam1997}{book}{
      author={Temam, Roger},
       title={Infinite-dimensional dynamical systems in mechanics and
  physics.},
    language={English},
     edition={2nd ed.},
      series={Appl. Math. Sci.},
   publisher={New York, NY: Springer},
        date={1997},
      volume={68},
        ISBN={0-387-94866-X},
}

\bib{Temam2001}{book}{
      author={Temam, Roger},
       title={Navier-{S}tokes equations},
   publisher={AMS Chelsea Publishing, Providence, RI},
        date={2001},
        ISBN={0-8218-2737-5},
         url={https://doi.org/10.1090/chel/343},
        note={Theory and numerical analysis, Reprint of the 1984 edition},
}

\end{biblist}
\end{bibdiv}
	
\end{document}